\documentclass[11pt]{article}%
\usepackage{amsmath,mathrsfs}
\usepackage{amsfonts}
\usepackage{amssymb}
\usepackage{graphicx,subfigure,color}
\usepackage{algorithm}
\usepackage{multirow}
\usepackage{hyperref}
\usepackage{xcolor}
\usepackage[affil-it]{authblk}
\usepackage{setspace}
\usepackage[numbers,sort&compress]{natbib}
\usepackage{subfigure}
\usepackage{booktabs}
\usepackage{latexsym,array,graphicx}
\providecommand{\U}[1]{\protect\rule{.1in}{.1in}}
\usepackage{array}
\usepackage{comment}
\usepackage{algorithm,algpseudocode,float}
\usepackage[toc,page]{appendix}
\usepackage{hyperref}


\numberwithin{equation}{section}
\newenvironment{keywords}{\noindent{\bf Key words:}}

\newenvironment{AMS}{\noindent{\bf Mathematics Subject Classification:}}

\newtheorem{theorem}{Theorem}
\newtheorem{example}{Example}
\newtheorem{lemma}[theorem]{Lemma}
\newtheorem{property}{Property}

\newtheorem{proposition}{Proposition}
\newtheorem{remark}{Remark}

\newcommand{\diag}{\mathrm{diag}}
\newcommand{\blockdiag}{\mathrm{blkdiag}}

\newenvironment{proof}{\noindent{\bf Proof:}}{\hfill\fbox{}\vspace*{1mm}}

\usepackage{authblk}
\title{A single-sided all-at-once preconditioning for linear system from a non-local evolutionary equation with weakly singular kernels}

\begin{document}	
	
	\author[a]{Xuelei Lin \thanks{E-mail: linxuelei@hit.edu.cn}}
	\author[b]{Jiamei Dong\thanks{Corresponding author:21482799@life.hkbu.edu.hk}}
	\author[b]{Sean Hon}
	\affil[a]{School of Science, Harbin Institute of Technology,
		Shenzhen 518055, China}
	\affil[b]{Department of Mathematics, Hong Kong Baptist University, Kowloon Tong, Hong Kong SAR, People’s Republic China}

	\maketitle
	
	\begin{abstract}
		{In 	[X. L. Lin, M. K. Ng, and Y. Zhi. {\it J. Comput. Phys.}, 434 (2021), pp. 110221] and [Y. L. Zhao, J. Wu, X. M. Gu, and H. Li. {\it Comput. Math. Appl.}, 148(2023), pp. 200--210]}, two-sided preconditioning techniques are proposed for non-local evolutionary equations, which possesses (i) mesh-size independent theoretical bound of condition number of the two-sided preconditioned matrix; (ii) small and stable iteration numbers in numerical tests. In this paper, we modify the two-sided preconditioning  by multiplying the left-sided and the right-sided preconditioners together as a single-sided preconditioner. Such a  single-sided preconditioner essentially derives from approximating the spatial matrix with a fast diagonalizable matrix and keeping the temporal matrix unchanged. Clearly, the matrix-vector multiplication of the single-sided preconditioning is faster to compute than that of the two-sided one, since the single-sided preconditioned matrix has a simpler structure. More importantly, we show theoretically that the single-sided preconditioned generalized minimal residual (GMRES) method has a convergence rate no worse than the two-sided preconditioned one. As a result, the one-sided preconditioned GMRES solver requires less computational time than the two-sided preconditioned GMRES solver in total. 
		Numerical results are reported to show the efficiency of the proposed single-sided preconditioning technique.
	\end{abstract}
	\begin{keywords} Preconditioning, Krylov subspace, all-at-once system, Toeplitz matrix, parallel-in-time
	\end{keywords}
	
	\begin{AMS}
		65F10; 65F08; 15A12; 15A60
	\end{AMS}
	
	\section{Introduction}\label{introduction}
	
	Consider a {space and time fractional} Bloch-Torrey equation (see, e.g., \cite{zhou2010studies,magin2008anomalous,zhao2023bilateral}), which is a non-local evolutionary equation with weakly singular kernels as follows 
	\begin{equation}
		\begin{cases}
			\frac{1}{\Gamma(1-\alpha)}\int_{0}^{t}\frac{\partial u({\bf x},s)}{\partial s}(t-s)^{-\alpha}ds=\sum\limits_{i=1}^{d}c_i\frac{\partial^{\beta_{i}}u({\bf x},t)}{\partial|x_{i}|^{\beta_{i}}}+f({\bf x},t),\quad {\bf x}\in \Omega\subset\mathbb{R}^d,~ t\in(0,T],\\
			u({\bf x},t)=0,\quad {\bf x}\in\partial \Omega,~t\in(0,T],\\
			u({\bf x},0)=\psi({\bf x}),\quad {\bf x}\in \Omega,\label{tsfde}
		\end{cases}
	\end{equation}
	where $\Gamma(\cdot)$ is the Gamma function, $0< \alpha< 1$,  $f$ and $\psi$ are both given functions; the boundary of $\Omega$ is $\partial\Omega$; $\Omega=\prod_{i=1}^{d}(\check{a}_i,\hat{a}_i)$;  ${\bf x}=(x_1,x_2,...,x_d)$ is a point in $\mathbb{R}^d$; for $i=1,\dots,d$, $c_{i}$  are positive constants; $\frac{\partial^{\beta_i}u({\bf x},t)}{\partial|x_i|^{\beta_i}}$ is the Riesz fractional derivative of order ${\beta_i}\in(1,2)$ with respect to $x_i$ defined as
	\begin{equation}\label{Riesz}
		\frac{\partial^{\beta_i}u({\bf x},t)}{\partial|x_i|^{\beta_i}}:=\frac{-1}{2\cos(\beta_i\pi/2)\Gamma(2-\beta_i)} \frac{\partial^2}{\partial x_i^2}\int_{\check{a}_i}^{\hat{a}_i}\frac{u(x_1,x_2,...,x_{i-1},\xi,x_{i+1},...,x_d,t)}{|x_i-\xi|^{\beta_i-1}}d\xi.\
	\end{equation}
	
	


	Since the analytical solution of \eqref{tsfde} is usually unavailable, discretization schemes are proposed for numerical solution; see, e.g., \cite{zhao2016galerkin,feng2016finite,arshad2017trapezoidal,hamid2019innovative,sun2016some,zhu2018high,bu2015finite,chen2013superlinearly,song2014spatially,yu2012computationally,yu2013numerical}.  Due to the nonlocal properties of the fractional differential operators, linear systems arising from these schemes are usually dense, which makes direct solvers like Gaussian elimination time-consuming and thus impractical to use.  Fortunately, due to the shift-invariant kernel involved in the fractional differential operators, the dense matrices have Toeplitz-like structure, whose matrix-vector multiplications can be fast implemented by fast Fourier transforms (FFTs). Based on the fast matrix-vector multiplication, iterative solvers are developed for solving the linear systems.

	The solvers for the linear systems are mainly classified as two types: time-stepping type and all-at-once type. The time-stepping solvers solve the linear system in a time-sequential manner (i.e., the unknowns at $n$-th time level have to be computed before computing the unknowns at $n+1$-th time level). For time-stepping type solvers, one may refer to \cite{lin2019fast,jiang2017fast,baffet2017kernel}. The all-at-once type solvers aim to solve the time-space linear system and update the unknowns at all time levels in a parallel manner; see, e.g.,  \cite{gu2020parallel,zhao2019limited,lin2018separable,zhao2021preconditioning}.
	Recently, new all-at-once type methods are studied in  {\cite{zhao2023bilateral,lin2021parallel}}, in which two-sided preconditioning techniques are studied for time-space linear systems arising from time-fractional diffusion equations. In the two-sided preconditioning method,  the spatial discretization matrix is  approximated by a $\tau$-matrix (i.e., a matrix is diagonalizable by sine transform, see, e.g., \cite{huang2021spectral}) and the temporal discretization matrix keeps unchanged. Then, an inverse of the square root of the { $\tau$-matrix} is spitted from the $\tau$-matrix approximation matrix as the right side preconditioner and the remaining part is taken as the left preconditioner. This is the so-called two-sided preconditioning method. The condition numbers of the two-sided preconditioned matrices are proven to be uniformly bounded by constants independent of discretization step sizes; see, e.g., \cite{zhao2023bilateral,lin2021parallel}. Numerical results in \cite{zhao2023bilateral,lin2021parallel} have shown the efficiency of the two-sided preconditioning method in terms of small and stable iteration numbers.
	
	Nevertheless, we find that the two-sided splitting is not necessary. In this paper, we show theoretically that the GMRES solver for the single-sided preconditioned system converges no slower than the GMRES solver for the two-sided preconditioned system. 
	
	Compared with the two-sided preconditioned system, solving the problem on the single-sided preconditioned system has two advantages: (i) computing the matrix-vector product associated with the single-sided preconditioned matrix requires less operations and is easier to implement (i.e., the code is simpler); (ii) the single-sided preconditioning is more flexible in the sense that the $\tau$ matrix can be replaced by other spatial approximations without a fast computable square root. 
	
	Indeed, the numerical results in the later section further show that the single-sided preconditioning method is more efficient than the two-sided preconditioning method in terms of computational time and iteration numbers, which supports our theoretical results.

	The contribution of this paper is twofold. Firstly, we significantly modify the two-sided preconditioning method by simplifying it to the single-sided preconditioning method that is easier to implement and requires less computational time. Secondly, the paper is not just focused on a single spatial discretization scheme. Actually, the theoretical results are established on common properties shared by several spatial discretization schemes. Hence, the proposed preconditioning method can be used to efficiently solve several  discrete problems arising from several  spatial discretization schemes in the literature, which demonstrates the range of its applicability.

	The paper is organized as follows. In Section \ref{discretization}, the discretization of \eqref{tsfde} and the corresponding linear system are presented. In Section 3, the $\tau$-preconditioner is proposed and the convergence behavior of the GMRES solver for the preconditioned system is discussed. In Section \ref{implementsection}, a fast implementation of the matrix-vector product of the preconditioned matrix is presented. In Section \ref{sec:numerical}, numerical results are reported for showing the performance of the proposed preconditioning method. Section \ref{sec:conclusion}  gives a conclusion.
	
	\section{Discretization of Eqn \eqref{tsfde} and the time-space linear system}\label{discretization}
	For any nonnegative integer $m$, $n$ with $m\leq n$, define the set $m\wedge n:=\{m,m+1,...,n-1,n\}$. Denote $\mathbb{N}^{+}$ be the set of all positive integers.
	\subsection{Temporal discretization}
	Let $N\in\mathbb{N}^{+}$ and the temporal step size $\mu=T/N$. With the L1 scheme (see, e.g.,{ \cite{sun2006fully,lin2007finite,jin2016analysis,liao2018sharp,zhao2019limited})}, the temporal discretization has the following form
	\begin{align}
		&	\frac{1}{\Gamma(1-\alpha)}\int_{0}^{t}\frac{\partial u({\bf x},s)}{\partial s}(t-s)^{-\alpha}ds\Big|_{t=n\mu}= \sum\limits_{k=1}^{n}l_{n-k}^{(\alpha)}u({\bf x},n\mu)+l^{(n,\alpha)}\psi({\bf x})+\mathcal{O}(\mu^{2-\alpha}),\label{convlqdra}
	\end{align}
	where ${\bf x}$ $ \in \Omega$, $n\in 1\wedge N$ and 
	\begin{align*}
		&l_{k}^{(\alpha)}=\begin{cases}
			[\Gamma(2-\alpha)\mu^{\alpha}]^{-1},\quad k=0,\\
			[\Gamma(2-\alpha)\mu^{\alpha}]^{-1}[(k+1)^{1-\alpha}-2k^{1-\alpha}+(k-1)^{1-\alpha}],\quad k\in 1\wedge (N-1),
		\end{cases}\notag\\
		&l^{(n,\alpha)}=[(n-1)^{1-\alpha}-n^{1-\alpha}][\Gamma(2-\alpha)\mu^{\alpha}]^{-1},\quad n\in 1\wedge N.\notag 
	\end{align*}
	
	\subsection{Space discretization}
	Actually, the theoretical results and algorithmic implementation in the later sections are applicable to a number of numerical schemes for the spatial discretization. Hence, in this subsection, we will present a general form of the spatial discretization and some common properties which will be utilized in the theoretical discussion.
	
	Let $m_i\in\mathbb{N}^{+}$, $i\in 1\wedge d$, and the $i$-th direction space step size $h_i=(\hat{a}_i-\check{a}_i)/(m_i+1)$. Setting the $j$ point along the $i$-th direction $x_{i,j}=\check{a}_i+jh_i$, for $j \in 0\wedge m_i+1$. In general, the discretization of $\frac{\partial^{\beta_i}u({\bf x},t)}{\partial|x_i|^{\beta_i}}$ on the grid points $x_{i,j}$ ($i=1,2,...,d,j=1,2,...,m_i$) is of the following form (see, e.g., \cite{huang2021preconditioner,lin2022tau,tian2015class,ding2017high,meerschaert2006finite,chen2014fourth}),
	\begin{align}
		&	\frac{\partial^{\beta_i}u({\bf x},t)}{\partial|x_i|^{\beta_i}}\Big|_{{\bf x}=(x_{1,j_1},x_{2,j_2},...,x_{d,j_d})}\approxeq -\frac{1}{h_i^{\beta_i}}\sum\limits_{k=1}^{m_i}w_{|j_i-k|}^{(\beta_i)}u(x_{1,j_1},...,x_{i-1,j_{i-1}},x_{i,k},x_{i+1,j_{i+1}},...,x_{d,j_d},t),	\label{spatialapp}
	\end{align}
	where $j_i\in 1\wedge m_i$.
	
	Some  common properties of the numbers $w_{k}^{(\beta)}$, $\beta\in(1,2)$ are listed as follows, which will be utilized in the later theoretical analysis. 
	There are several spatial discretization scheme possesses the following properties (see the discussion in Appendix).

	\begin{property}\label{wkprop}
		For $\beta\in(1,2)$, it holds\footnote{See the spatial discretization schemes satisfying  property 1 discussed in the Appendices}
		\begin{description}
			\item[(i)]$w_0^{(\beta)}>0$, $w_{k}^{(\beta)}\leq 0$ for $k\geq 1$;
			\item[(ii)] $\inf\limits_{m\geq 1}{(m+1)^{\beta}}\left(w_0^{(\beta)}+2\sum\limits_{k=1}^{m-1}w_k^{(\beta)}\right)>0$;
			\item[(iii)]$w_{k}^{(\beta)}\leq w_{k+1}^{(\beta)}$ for $k\geq 1$.
		\end{description}
	\end{property}
	
	\subsection{All-at-once system}
	Denote
	\begin{equation*}
		J=\prod\limits_{i=1}^{d}m_i,\quad m_1^{-}=m_d^{+}=1,\quad m_i^{-}=\prod\limits_{j=1}^{i-1}m_i,~i\in 2\wedge d,\quad  m_k^{+}=\prod\limits_{j=k+1}^{d}m_j,~k\in 1\wedge (d-1).
	\end{equation*}
	With Eqn \eqref{convlqdra} and Eqn \eqref{spatialapp} , we obtain the following all-at-once linear system as a discretization of the continuous problem \eqref{tsfde}
	\begin{equation}\label{aaosystem}
		{\bf A}{\bf u}={\bf f},
	\end{equation}
	where ${\bf u}=({\bf u}_1;{\bf u}_2;\dots;{\bf u}_{m_1}) \in \mathbb{R}^{NJ\times 1}$, ${\bf f}=({\bf f}_1;{\bf f}_2;\dots;{\bf f}_{m_1}) \in \mathbb{R}^{NJ\times 1}$. ${\bf u}_j$ is a vector component by lexicographic order approximate values of $u(\check{a}_1+jh_1,\cdot)$, $j\in 1\wedge m_1$, while ${\bf f}_j$ contains the initial values and the approximate values of $f(\check{a}_1+jh_1,\cdot)$ on the spatial grid points (or $f(\check{a}_1+(j+1/2)h_1,\cdot)$ in Eqn \eqref{aaosystem}); 
	\begin{equation}
		\begin{aligned}
			&{\bf A}={\bf B}\otimes {\bf I}_N+{\bf I}_{J}\otimes {\bf T},\\
			&{\bf B}=\sum\limits_{i=1}^{d}\eta_{i}{\bf I}_{m_i^{-}}\otimes {\bf W}_{m_i}^{(\beta_i)}\otimes {\bf I}_{m_i^{+}},\quad \eta_{i}=\frac{c_{i}}{h_i^{\beta_i}},\label{allatonce}
		\end{aligned}
	\end{equation}
	`$\otimes$' denotes the Kronecker product; for any $k\in\mathbb{N}^{+}$, ${\bf I}_{k}$ denotes a $k\times k$ identity matrix; the lower triangular toeplitz matrix  ${\bf T}\in\mathbb{R}^{N\times N}$ denotes the temporal discretization matrix, its first column is $(l_{0}^{(\alpha)},l_{1}^{(\alpha)},\dots,l_{N-1}^{(\alpha)})^T$.
	For any $\beta\in(1,2)$ and any $m\in\mathbb{N}^{+}$, the symmetric Toeplitz matrix ${\bf W}_{m}^{(\beta)}\in\mathbb{R}^{m\times m}$ formed as
	\begin{equation*}
		{\bf W}_{m}^{(\beta)}=\left[
		\begin{array}
			[c]{ccccc}
			w_0^{(\beta)} & w_{1}^{(\beta)} &\ldots   &w_{m-2}^{(\beta)}  &w_{m-1}^{(\beta)}\\
			w_1^{(\beta)} & w_0^{(\beta)}&w_{1}^{(\beta)}   &\ldots &w_{m-2}^{(\beta)}\\
			\vdots&\ddots &\ddots&\ddots&\vdots\\
			w_{m-2}^{(\beta)} & \ldots&w_1^{(\beta)}& w_0^{(\beta)} &w_{1}^{(\beta)}\\
			w_{m-1}^{(\beta)} & w_{m-2}^{(\beta)}& \ldots & w_1^{(\beta)} & w_{0}^{(\beta)}
		\end{array}
		\right].
	\end{equation*}
	
	\begin{lemma}\label{wpo}
		For any $\beta\in(1,2)$ and any  $m\geq 1$,    $	{\bf W}_{m}^{(\beta)}$ is a positive definite matrix.
	\end{lemma}
	\begin{proof}
		By Property \ref{wkprop} ${\bf (i)}-{\bf (ii)}$ and definition of $	{\bf W}_{m}^{(\beta)}$,  we have
		$$ \sum_{j=1,i\neq j}^{m}|	{\bf W}_{m}^{(\beta)}(i,j)|= \sum_{j=1}^{m-1} |w_j^{(\beta)}|<-2\sum\limits_{k=1}^{m-1}w_k^{(\beta)}<|w_0^{(\beta)} |={\bf W}_{m}^{(\beta)}(i,i),\ \text{ for} \  i=1,\cdots m.  $$
		Combining with Gershgorin circle theorem \footnote{see the definition of Gershgorin circle theorem in \cite{goluvan2013}}, 	Matrix $	{\bf W}_{m}^{(\beta)}$ is a positive definite matrix.
	\end{proof}
	Although matrix ${\bf W}_m^{(\beta)}$ and matrix ${\bf B}$ are positive definite symmetric matrices, the matrix ${\bf A}$ is an unsymmetric matrix due to matrix ${\bf T}$ is a lower tridiagonal matrix. Therefore, some popular iterative methods, like the conjugate gradient method and the Minimal Residual Method, are not applicable. We use the GMRES  method to solve this general nonsymmetric linear system.
	\section{\textbf{{The $\tau$ preconditioner and the convergence behavior of GMRES solver for the preconditioned system}}}
	In this section, our $\tau$ preconditioner is proposed for  the linear system \eqref{aaosystem}, which is used as a single-sided preconditioner. We employ GMRES solver to solve the single-sided preconditioned system. To investigate the convergence behavior of GMRES solver for the single-sided preconditioned system, we consider an auxiliary two-sided preconditioned system. We show that (i) GMRES solver for the single-sided preconditioned system has a convergence rate no worse than that for the auxiliary two-sided preconditioned system; (ii) the auxiliary two-sided preconditioned system has a condition number uniformly bounded by 3.
	
	We first define some notations that will be used later.
	For any real symmetric positive semi-definite matrix ${\bf H}\in\mathbb{R}^{k\times k}$, define $${\bf H}^{z}:={\bf V}^{\rm T}\diag[({\bf D}(i,i))^z]_{i=1}^{k}{\bf V},\quad z\in\mathbb{R},$$ where ${\bf H}={\bf V}^{\rm T}{\bf D}{\bf V}$ denotes the orthogonal diagonalization of ${\bf H}$.
	For a symmetric Toeplitz matrix ${\bf T}_m\in\mathbb{R}^{m\times m}$ with $(t_1,t_2,...,t_m)^{\rm T}\in\mathbb{R}^{m\times 1}$, define its $\tau$-matrix approximation as
	\begin{equation}\label{tauopdef}
		\tau({\bf T}_m):={\bf T}_m-{\bf H}_m,
	\end{equation}
	where ${\bf H}_m$ is a Hankel matrix with $(t_3,t_4,...,t_m,0,0)^{\rm T}$ as its first column and $(0,0,t_m,...,t_4,t_3)^{\rm T}$ as its last column. An interesting property of the $\tau$-matrix defined in \eqref{tauopdef}  is that it is diagonalizable by sine transform matrix, i.e.,
	\begin{equation}\label{taumatdiag}
		\tau({\bf T}_m)={\bf S}_m{\bf Q}{\bf S}_m,
	\end{equation}
	where ${\bf Q}={\rm diag}(q_i)_{i=1}^{m}$ is a diagonal matrix with
	\begin{equation}\label{sigmicomp}
		q_{i}=t_1+2\sum\limits_{j=2}^{m}t_j\cos\left(\frac{\pi i(j-1)}{m+1}\right),\quad  i\in 1\wedge m.
	\end{equation} 
	\begin{equation}\label{sinematdef}
		{\bf S}_{m}:= \left[\sqrt{\frac{2}{m+1}}\sin\left(\frac{\pi jk}{m+1}\right)\right]_{j,k=1}^{m}
	\end{equation}
	is a sine transform matrix. It is easy to verify that ${\bf S}_m$ is a real symmetric orthogonal matrix, i.e., ${\bf S}_m={\bf S}_m^{\rm T}={\bf S}_{m}^{-1}$. The product between matrix ${\bf S}_m$ and a given vector of length $m$ can be fast computed within $\mathcal{O}(m\log m)$ operations using { fast sine transform (FST)}. Let ${\bf e}_{m,i}\in\mathbb{R}^{m\times 1}$ denotes the $i$th column of the $m\times m$ identity matrix.
	We also note that the $m$ numbers $\{q_i\}_{i=1}^{m}$ defined in \eqref{sigmicomp} can be computed by
	\begin{equation*}
		(q_1;q_2;\cdots;q_m)={\rm diag}({\bf S}_m{\bf e}_{m,1})^{-1}[{\bf S}_m\tau({\bf T}_m){\bf e}_{m,1}].
	\end{equation*}
	From the equality above, we see that the computation of $\{q_i\}_{i=1}^{m}$ requires only a few FST, which requires $\mathcal{O}(m\log m)$ operations.
	
	With Eqn \eqref{taumatdiag}, we define an approximation ${\bf B}_{\tau}$ to ${\bf B}$ as follows
	\begin{equation}\label{btaudef}
		{\bf B}_{\tau}=\frac{\sqrt{3}}{2}\sum\limits_{i=1}^{d}\eta_i{\bf I}_{m_i^{-}}\otimes \tau({\bf W}_{m_i}^{(\beta_i)})\otimes {\bf I}_{m_i^{+}}.
	\end{equation}
	The reason why we put a factor $\frac{\sqrt{3}}{2}$ in the above definition will be explained in Subsection \ref{mainrsltsec}.
	From Eqn \eqref{taumatdiag}, \eqref{sigmicomp} and properties of the single-dimension sine transform matrix ${\bf S}_m$, we see that ${\bf B}_{\tau}$ is diagonalizable by a multi-dimension sine transform matrix, i.e.,
	\begin{align}
		&{\bf B}_{\tau}=\frac{\sqrt{3}}{2}{\bf S}{\bf \Lambda}{\bf S},\label{btaudiag}\\
		&{\bf S}:=\bigotimes\limits_{i=1}^{d}{\bf S}_{m_i},\quad {\bf \Lambda}=\sum\limits_{i=1}^{d}{\bf I}_{m_i^{-}}\otimes{\bf \Lambda}_i\otimes {\bf I}_{m_i^{+}},\quad {\bf \Lambda}_i={\rm diag}(\lambda_{i,j})_{j=1}^{m_i},\notag\\
		&\lambda_{i,j}=\frac{1}{h^{\beta_i}}w_{0}^{(\beta_i)}+\frac{2}{h^{\beta_i}}\sum\limits_{k=2}^{m_i}w_{k-1}^{(\beta_i)}\cos\left(\frac{\pi j(k-1)}{m_i+1}\right)\label{lambdaijdef}.
	\end{align}
	
	Then, our $\tau$ preconditioner is defined as
	\begin{equation}\label{pdef}
		{\bf P}:={\bf I}_J\otimes{\bf T}+{\bf B}_{\tau}\otimes {\bf I}_N.
	\end{equation}
	Instead of solving \eqref{aaosystem}, we employ the GMRES solver to solve the following equivalent system  by
	\begin{equation}\label{prewhole}
		{\bf P}^{-1}{\bf A}{\bf u}= {\bf P}^{-1}{\bf f}.
	\end{equation}
	
	\subsection{Convergence behavior of GMRES solver for system \eqref{prewhole}} \label{mainrsltsec}
	This section introduces a two-sided variant of \eqref{prewhole} as an auxiliary problem (see \eqref{twosidedpreclinsys}), which helps us understand the convergence behavior of GMRES solver for \eqref{prewhole}. 
	
	It is straightforward to verify that the solution of \eqref{prewhole} is equivalent to the following two steps.
	\begin{align}
		&{\bf P}_l^{-1}{\bf A}{\bf P}_r^{-1}\hat{\bf u}={\bf P}_l^{-1}{\bf f},\label{twosidedpreclinsys}
	\end{align}
	where the solution of \eqref{prewhole} and the solution of \eqref{twosidedpreclinsys} are related by ${\bf u}={\bf P}_{r}^{-1}\hat{\bf u}$; 
	\begin{equation*}
		{\bf P}_l={\bf B}_{\tau}^{-\frac{1}{2}}\otimes{\bf T}+{\bf B}_{\tau}^{\frac{1}{2}}\otimes{\bf I}_N,\quad {\bf P}_r={\bf B}_{\tau}^{\frac{1}{2}}\otimes{\bf I}_N.
	\end{equation*}
	
	In this section, we will show that (i) the GMRES solver for \eqref{prewhole} converges no slower than that for \eqref{twosidedpreclinsys}; (ii) \eqref{twosidedpreclinsys} is a well-conditioned system, i.e., $\sup\limits_{N,J>0}\kappa_2({\bf P}_l^{-1}{\bf A}{\bf P}_r^{-1})<+\infty$. Here, $\kappa_2({\bf C}):=||{\bf C}||_2||{\bf C}^{-1}||_2$ is called condition number of an invertible  matrix ${\bf C}$. Typically, Krylov subspace solvers converge quickly for well-conditioned system. Indeed, the numerical results in \cite{lin2021parallel,zhao2023bilateral} show that GMRES solver for the two-sided preconditioned system converges in a few iteration number no matter how large the matrix size is. Hence, a theoretical result that GMRES solver for \eqref{prewhole} converges no slower than that for \eqref{twosidedpreclinsys} guarantees that the GMRES solver for \eqref{prewhole} is a fast solver which further improves the iteration number and computational time compared with the two-sided preconditioning method.

	For any Hermitian matrices ${\bf C}_1,{\bf C}_2\in\mathbb{R}^{k\times k}$, denote ${\bf C}_2 \succ ({\rm or} \succeq) \ {\bf C}_1$
	if ${\bf C}_2-{\bf C}_1$ is positive definite (or semi-definite). Especially, we denote ${\bf C}_2 \succ ({\rm or} \succeq) \ {\bf O}$, if ${\bf C}_2$ itself is positive definite (or semi-definite).
	Also, ${\bf C}_1 \prec ({\rm or} \preceq) \ {\bf C}_2$ and ${\bf O} \prec ({\rm or} \preceq) \ {\bf C}_2$  have the same meanings as those of ${\bf C}_2 \succ ({\rm or} \succeq) \ {\bf C}_1$ and ${\bf C}_2 \succ ({\rm or} \succeq) \ {\bf O}$, respectively.
	
	For a Hermitian matrix ${\bf H}$, denote by $\lambda_{\min}({\bf H})$ and $\lambda_{\max}({\bf H})$, the minimal and the maximal eigenvalues of ${\bf H}$, respectively.
	\begin{proposition}\label{btauspdprops}
		The Matrices ${\bf \Lambda}\succ {\bf O}$,  ${\bf B}_{\tau}\succ {\bf O}$ and ${\bf B}\succ{\bf O}$. And $\inf\limits_{J>0}\lambda_{\min}({\bf B}_{\tau})\geq \check{c}>0,$ where
		$$
		\check{c}:=\frac{\sqrt{3}}{2}\sum\limits_{i=1}^{d}\frac{1}{(\hat{a}_i-\check{a}_i)^{\beta_i}}\inf \limits_{m\geq 1}(m+1)^{\beta_i}
		\left(w_{0}^{(\beta_i)}+2\sum\limits_{k=1}^{m-1}w_{k}^{(\beta_i)}\right).
		$$
	\end{proposition}
	\begin{proof}
		By equation \eqref{btaudiag}, we know that the spectrum of ${\bf B}_{\tau}$ consists of
		\begin{equation*}
			\mathscr{S}:=\left\{\frac{\sqrt{3}}{2}\sum\limits_{i=1}^{d}\lambda_{i,j_i}\bigg| 1\leq j_i\leq m_i,~1\leq i\leq d \right\}.
		\end{equation*}
		By \eqref{lambdaijdef} and Property \ref{wkprop}, we have
		\begin{align*}
			\lambda_{i,j_i}&=\frac{1}{h_i^{\beta_i}}w_{0}^{(\beta_i)}+\frac{2}{h_i^{\beta_i}}\sum\limits_{k=2}^{m_i}w_{k-1}^{(\beta_i)}\cos\left(\frac{\pi j_i(k-1)}{m_i+1}\right)\\
			&\geq \frac{1}{h_i^{\beta_i}}w_{0}^{(\beta_i)}+\frac{2}{h_i^{\beta_i}}\sum\limits_{k=2}^{m_i}w_{k-1}^{(\beta_i)}\\
			&=\frac{1}{(\hat{a}_i-\check{a}_i)^{\beta_i}}\times (m_i+1)^{\beta_i}
			\left(w_{0}^{(\beta_i)}+2\sum\limits_{k=1}^{m_i-1}w_{k}^{(\beta_i)}\right)\\
			&\geq \frac{1}{(\hat{a}_i-\check{a}_i)^{\beta_i}}\times\inf \limits_{m\geq 1}(m+1)^{\beta_i}
			\left(w_{0}^{(\beta_i)}+2\sum\limits_{k=1}^{m-1}w_{k}^{(\beta_i)}\right)\\
			&>0.
		\end{align*}
		Hence,
		\begin{equation*}
			\lambda_{\min}({\bf B}_{\tau})=\min\limits_{y\in\mathscr{S}}y\geq \frac{\sqrt{3}}{2}\sum\limits_{i=1}^{d}\frac{1}{(\hat{a}_i-\check{a}_i)^{\beta_i}}\inf \limits_{m\geq 1}(m+1)^{\beta_i}
			\left(w_{0}^{(\beta_i)}+2\sum\limits_{k=1}^{m-1}w_{k}^{(\beta_i)}\right)=\check{c}.
		\end{equation*}
		As the inequality above holds for arbitrary $J>0$, 
		\begin{equation*}
			\inf\limits_{J>0}\lambda_{\min}({\bf B}_{\tau})\geq \check{c}.
		\end{equation*}
		As indicated by Property \ref{wkprop}${\bf (ii)}$, $\inf \limits_{m\geq 1}(m+1)^{\beta_i}
		\left(w_{0}^{(\beta_i)}+2\sum\limits_{k=1}^{m-1}w_{k}^{(\beta_i)}\right)>0$ for each $i$. Therefore, $\check{c}>0$.
		Hence, ${\bf \Lambda }\succ {\bf O}$ and ${\bf B}_{\tau}\succ {\bf O}$.
		
		Finally, with Lemma \ref{wpo}, it is easy to see that ${\bf B}\succ {\bf O}$. The proof is completed.
	\end{proof}

	The convergence behavior of GMRES is closely related to the Krylov subspace. For a square matrix $\mathbf{E}\in \mathbb{R}^{m \times m}$ and a vector ${\bf x}\in \mathbb{R}^{m\times 1}$, a Krylov subspace of degree $j\geq 1$ is defined as follows
	$$\mathcal{K}_j(\mathbf{E},{\bf x}):=\text{span}\{{\bf x},\mathbf{E}{\bf x},\mathbf{E}^2{\bf x},\dots,\mathbf{E}^{j-1}{\bf x} \}.$$
	For a set $\mathcal{S}$ and a point $z$, we denote
	$$z+\mathcal{S}:=\{z+x|x\in \mathcal{S}\}.$$
	We recall the relation between the iterative solution by GMRES and the Krylov subspace in the following lemma.
	\begin{lemma}\textnormal{(see, e.g., \cite{saad2003iterative})}\label{krylov}
		For a non-singular $m\times m$ real linear system ${\bf Z} y = {\bf b}$, let $y_j$ be the iterative solution by GMRES at $j$-th ($j\geq1$) iteration step with $y_0$ as initial guess. Then, the $j$-th iteration solution $y_j$ minimize the residual error over the Krylov subspace $\mathcal{K}_j({\bf Z},{\bf r}_0)$ with ${\bf r}_0= {\bf b}-{\bf Z}y_0,$ i.e.,	
		$$
		\mathbf{y}_{j}=\underset{\mathbf{v} \in \mathbf{y}_{0}+\mathcal{K}_{j}\left(\mathbf{Z}, \mathbf{r}_{0}\right)}{\arg \min }\|\mathbf{b}-\mathbf{Z} \mathbf{v}\|_{2} .
		$$
	\end{lemma}

	\begin{theorem}\label{residual}
		Let $\hat{\mathbf{u}}_{0}$ be the initial guess for \eqref{twosidedpreclinsys}. Let $\mathbf{u}_{0}:=\mathbf{P}_{r}^{-1} \hat{\mathbf{u}}_{0}$ be the initial guess for \eqref{prewhole}. Let $\mathbf{u}_{j}\left(\hat{\mathbf{u}}_{j}\right.$, respectively$)$ be the $j$-th $(j \geq 1)$ iteration solution derived by applying GMRES solver to \eqref{prewhole} (\eqref{twosidedpreclinsys}, respectively) with $\mathbf{u}_{0}\left(\hat{\mathbf{u}}_{0}\right.$, respectively) as initial guess. Then,
		$$
		\left\|\mathbf{r}_{j}\right\|_{2} \leq \frac{1}{\sqrt{\check{c}}}\left\|\hat{\mathbf{r}}_{j}\right\|_{2}
		$$
		where $\mathbf{r}_{j}:=\mathbf{P}^{-1} \mathbf{f}-\mathbf{P}^{-1} \mathbf{A} \mathbf{u}_{j}\left(\hat{\mathbf{r}}_{j}:=\mathbf{P}_{l}^{-1} \mathbf{f}-\mathbf{P}_{l}^{-1} \mathbf{A} \mathbf{P}_{r}^{-1} \hat{\mathbf{u}}_{j}\right.$, respectively$)$ denotes the residual vector at $j$-th GMRES iteration for \eqref{prewhole} (\eqref{twosidedpreclinsys}, respectively); $\check{c}>0$ defined in Proposition \ref{btauspdprops}  is a constant independent of $N$ and $J$.
	\end{theorem}

	\begin{proof}
		By applying Lemma \ref{krylov} to \eqref{twosidedpreclinsys}, we see that
		$$
		\hat{\mathbf{u}}_{j}-\hat{\mathbf{u}}_{0} \in \mathcal{K}_{j}\left(\mathbf{P}_{l}^{-1} \mathbf{A P}_{r}^{-1}, \hat{\mathbf{r}}_{0}\right),
		$$
		where $\hat{\mathbf{r}}_{0}=\mathbf{P}_{l}^{-1} \mathbf{f}-\mathbf{P}_{l}^{-1} \mathbf{A P}_{r}^{-1} \hat{\mathbf{u}}_{0}$. Notice that $\left(\mathbf{P}_{l}^{-1} \mathbf{A P}_{r}^{-1}\right)^{k}=\mathbf{P}_{r}\left(\mathbf{P}^{-1} \mathbf{A}\right)^{k} \mathbf{P}_{r}^{-1}$ for each $k \geq 0$. Therefore,
		$$
		\begin{aligned}
			\mathcal{K}_{j}\left(\mathbf{P}_{l}^{-1} \mathbf{A} \mathbf{P}_{r}^{-1}, \hat{\mathbf{r}}_{0}\right) &=\operatorname{span}\left\{\left(\mathbf{P}_{l}^{-1} \mathbf{A} \mathbf{P}_{r}^{-1}\right)^{k}\left(\mathbf{P}_{l}^{-1} \mathbf{f}-\mathbf{P}_{l}^{-1} \mathbf{A} \mathbf{P}_{r}^{-1} \hat{\mathbf{u}}_{0}\right)\right\}_{k=0}^{j-1} \\
			&=\operatorname{span}\left\{\mathbf{P}_{r}\left(\mathbf{P}^{-1} \mathbf{A}\right)^{k} \mathbf{P}_{r}^{-1}\left(\mathbf{P}_{l}^{-1} \mathbf{f}-\mathbf{P}_{l}^{-1} \mathbf{A P}_{r}^{-1} \hat{\mathbf{u}}_{0}\right)\right\}_{k=0}^{j-1} \\
			&=\operatorname{span}\left\{\mathbf{P}_{r}\left(\mathbf{P}^{-1} \mathbf{A}\right)^{k}\left(\mathbf{P}^{-1} \mathbf{f}-\mathbf{P}^{-1} \mathbf{A} \mathbf{u}_{0}\right)\right\}_{k=0}^{j-1}
		\end{aligned}
		$$
		where the last equality comes from the facts that $\mathbf{P}=\mathbf{P}_{l} \mathbf{P}_{r}$ and that $\mathbf{u}_{0}=\mathbf{P}_{r}^{-1} \hat{\mathbf{u}}_{0}$.
		Therefore,
		$$
		\mathbf{P}_{r}^{-1} \hat{\mathbf{u}}_{j}-\mathbf{u}_{0}=\mathbf{P}_{r}^{-1}\left(\hat{\mathbf{u}}_{j}-\hat{\mathbf{u}}_{0}\right) \in \operatorname{span}\left\{\left(\mathbf{P}^{-1} \mathbf{A}\right)^{k}\left(\mathbf{P}^{-1} \mathbf{f}-\mathbf{P}^{-1} \mathbf{A} \mathbf{u}_{0}\right)\right\}_{k=0}^{j-1}=\mathcal{K}_{j}\left(\mathbf{P}^{-1} \mathbf{A}, \mathbf{r}_{0}\right),
		$$
		where $\mathbf{r}_{0}=\mathbf{P}^{-1} \mathbf{f}-\mathbf{P}^{-1} \mathbf{A u}_{0}$. In other words,
		$$
		\mathbf{P}_{r}^{-1} \hat{\mathbf{u}}_{j} \in \mathbf{u}_{0}+\mathcal{K}_{j}\left(\mathbf{P}^{-1} \mathbf{A}, \mathbf{r}_{0}\right) .
		$$
		By applying Lemma \ref{krylov} to \eqref{prewhole}, we know that
		$$
		\mathbf{u}_{j}=\underset{\mathbf{v} \in \mathbf{u}_{0}+\mathcal{K}_{j}\left(\mathbf{P}^{-1} \mathbf{A}, \mathbf{r}_{0}\right)}{\arg \min }\left\|\mathbf{P}^{-1} \mathbf{f}-\mathbf{P}^{-1} \mathbf{A} \mathbf{v}\right\|_{2} .
		$$
		Therefore,
		$$
		\begin{aligned}
			\left\|\mathbf{r}_{j}\right\|_{2}=\left\|\mathbf{P}^{-1} \mathbf{f}-\mathbf{P}^{-1} \mathbf{A} \mathbf{u}_{j}\right\|_{2} & \leq\left\|\mathbf{P}^{-1} \mathbf{f}-\mathbf{P}^{-1} \mathbf{A} \mathbf{P}_{r}^{-1} \hat{\mathbf{u}}_{j}\right\|_{2} \\
			&=\left\|\mathbf{P}_{r}^{-1} \hat{\mathbf{r}}_{j}\right\|_{2} \\
			&=\sqrt{\hat{\mathbf{r}}_{j}^{\mathrm{T}} \mathbf{P}_{r}^{-2} \hat{\mathbf{r}}_{j}} \\
			&=\sqrt{\hat{\mathbf{r}}_{j}^{\mathrm{T}}\left(\mathbf{B}_{\tau}^{-1} \otimes \mathbf{I}_{N}\right) \hat{\mathbf{r}}_{j}} \\
			& \leq \sqrt{\frac{1}{\check{c}} \hat{\mathbf{r}}_{j}^{\mathrm{T}} \hat{\mathbf{r}}_{j}}=\frac{1}{\sqrt{\check{c}}}\left\|\hat{\mathbf{r}}_{j}\right\|_{2},
		\end{aligned}
		$$
		where the last inequality comes from Proposition \ref{btauspdprops}.
	\end{proof}
	
	The residuals relationship between our single-sided preconditioner \eqref{prewhole} and the two-sided preconditioner method \eqref{twosidedpreclinsys} is proven in Theorem \ref{residual}. In the next, we estimate the condition number of the matrix ${\bf P}_{l}^{-1}{\bf A}{\bf P}_{r}^{-1}$ in \eqref{twosidedpreclinsys}.
	
	For a square matrix ${\bf C}$, denote by $\sigma({\bf C})$, the spectrum of ${\bf C}$.

	\begin{lemma}\label{taumatprecspectralm} 
		For any $\beta\in(1,2)$ and any $m\in\mathbb{N}^{+}$, $\sigma(\tau({\bf W}_{m}^{(\beta)})^{-1}{\bf W}_{m}^{(\beta)})\subset(1/2,3/2)$. 
	\end{lemma}
	\begin{proof}
		Denote ${\bf H}_{m}^{(\beta)}:={\bf W}_{m}^{(\beta)}-\tau({\bf W}_{m}^{(\beta)})$. Rewrite ${\bf H}_{m}^{(\beta)}$ as ${\bf H}_{m}^{(\beta)}=[h_{ij}]_{i,j=1}^{m}$. Then, straightforward calculation yields that
		\begin{equation*}
			h_{ij}=\begin{cases}
				w_{i+j}^{(\beta)},\quad i+j<m-1,\\
				w_{2m+2-(i+j)}^{(\beta)},\quad i+j> m+1,\\
				0,\quad  {\rm otherwise}.
			\end{cases}
		\end{equation*}
		By Property \ref{wkprop}${\bf (i)}$, we know that 
		\begin{equation}\label{hijneg}
			h_{ij}\leq 0,\quad 1\leq i,j\leq m.
		\end{equation}
		Denote $p_{ij}=w_{|i-j|}^{(\beta)}-h_{ij}$.		Then, \eqref{hijneg} and Property \ref{wkprop}${\bf (i)}$,${\bf (iii)}$ imply that
		\begin{equation}\label{pijsign}
			p_{ij}=\begin{cases}
				w_{0}^{(\beta)}-h_{ii}>0,\quad i=j,\\
				w_{|i-j|}^{(\beta)}-w_{i+j}^{(\beta)}\leq 0,\quad i+j<m-1{\rm~and~} i\neq j,\\
				w_{|i-j|}^{(\beta)}-w_{2m+2-(i+j)}^{(\beta)}\leq 0,\quad i+j> m+1{\rm~and~} i\neq j,\\
				w_{|i-j|}^{(\beta)}\leq 0,\quad {\rm otherwise}.
			\end{cases}
		\end{equation}
		Let $(\lambda,z)$ be an eigen-pair of $\tau({\bf W}_{m}^{(\beta)})^{-1}{\bf H}_{m}^{(\beta)}$ such that $||z||_{\infty}=1$. Then, 
		\begin{equation}\label{tauwinvweigpair}
			{\bf H}_{m}^{(\beta)} z = \lambda \tau({\bf W}_{m}^{(\beta)})z. 
		\end{equation}
		Rewrite $z$ as $z=(z_1,z_2,...,z_m)^{\rm T}$. 
		\eqref{tauwinvweigpair} implies that for each $i=1,2,...,m$, it holds
		\begin{equation*}
			\sum_{j=1}^{m}h_{ij}z_j=\lambda \sum_{j=1}^{m}p_{ij}z_j.
		\end{equation*}
		Hence,
		$$\lambda p_{ii}z_i=\sum_{j=1}^{m}h_{ij}z_j-\lambda \sum_{j=1,i\neq j}^{m}p_{ij}z_j,\quad i=1,2,...,m.$$
		As $||z||_{\infty}=1$, there exists  $k_0\in\{1,2,...,m\}$ such that $|z_{k_0}|=1$. Then,
		\begin{equation}\label{lambdapkkbd}
			|\lambda| |p_{kk}|\leq \sum_{j=1}^{m}|h_{kj}|+|\lambda| \sum_{j=1,k\neq j}^{m}|p_{kj}|.
		\end{equation}
		By \eqref{pijsign} and Property \ref{wkprop}${\bf (ii)}$, we have
		\begin{align*}
			&|p_{kk}|-\sum_{j=1,j\neq k}^{m}|p_{kj}|-2\sum_{j=1}^{m}|h_{kj}| \\
			=&(w_0^{(\beta)}-h_{kk})- \sum_{j=1,j\neq k}^{m}(h_{kj}-w_{|k-j|}^{(\beta)})+2\sum_{j=1}^{m}h_{kj}\\
			=& w_0^{(\beta)}+\sum_{j=1,j\neq k}^{m}w_{|k-j|}^{(\beta)}+\sum_{j=1}^{m}h_{kj} \\
			=& w_0^{(\beta)}+\left(\sum_{j=1}^{k-1}w_{j}^{(\beta)}+\sum_{j=1}^{m-k}w_{j}^{(\beta)}\right)+\left(\sum_{j=k+1}^{m-1}w_{j}^{(\beta)}+\sum_{j=m-k+2}^{m-1}w_{j}^{(\beta)}\right)\\
			\geq&w_0^{(\beta)}+2\sum_{j=1}^{m-1}w_{j}^{(\beta)}  > 0,
		\end{align*} 
		which combined with \eqref{lambdapkkbd} implies that
		$$|\lambda|\leq \frac{ \sum\limits_{j=1}^{m}|h_{kj}|}{|p_{kk}|-\sum\limits_{j=1,k\neq j}^{m}|p_{kj}|}\leq  \frac{ \sum\limits_{j=1}^{m}|h_{kj}|}{2\sum\limits_{j=1}^{m}|h_{kj}|} =\frac{1}{2}. $$
		Therefore,
		\begin{equation}\label{tauinvhspectr}
			\sigma(\tau({\bf W}_{m}^{(\beta)})^{-1}{\bf H}_{m}^{(\beta)})\subset(-1/2,1/2).
		\end{equation}
		
		Since	$	\tau({\bf  W}_m)={\bf  W}_m-{\bf H}_m$, $$\tau({\bf W}_{m}^{(\beta)})^{-1}{\bf W}_{m}^{(\beta)}=\tau({\bf W}_{m}^{(\beta)})^{-1}(\tau({\bf W}_{m}^{(\beta)})+{\bf H}_m^{(\beta)})={\bf I}+\tau({\bf W}_{m}^{(\beta)})^{-1}{\bf H}_m^{(\beta)},$$
		which means $\sigma(\tau({\bf W}_{m}^{(\beta)})^{-1}{\bf W}_{m}^{(\beta)})=1+\sigma(\tau({\bf W}_{m}^{(\beta)})^{-1}{\bf H}_m^{(\beta)})\subset(1/2,3/2)$. The proof is completed. 
	\end{proof}
	
	\begin{remark}
		The proof of Lemma \ref{taumatprecspectralm} is similar to that of \cite[Lemma 4.3]{huang2021spectral}. But the proof in \cite[Lemma 4.3]{huang2021spectral} is solely for the shifted Gr\"{u}nwald spatial discretization. When handling a new spatial scheme, repeating the proof is too redundant. Because of this reason, we present Lemma \ref{taumatprecspectralm}, a more general result, which tell us that to prove the spectrum inclusion $\sigma(\tau({\bf W}_{m}^{(\beta)})^{-1}{\bf W}_{m}^{(\beta)})\subset(1/2,3/2)$, it suffices to verify Property \ref{wkprop} of a spatial scheme.  
	\end{remark}
	
	With Lemma \ref{taumatprecspectralm}, one can immediately prove the following Proposition.
	\begin{proposition}\label{taumatneqs}
		For any $\beta\in(1,2)$ and $m\in\mathbb{N}^{+}$, it holds
		\begin{description}
			\item[(i)]$ {\bf O}\prec\frac{\sqrt{3}}{3}{\bf B}_{\tau}\preceq {\bf B}\preceq\sqrt{3}{\bf B}_{\tau}$;
			\item [(ii)] $\frac{\sqrt{3}}{3}{\bf I}_{J}\preceq{\bf B}^{\frac{1}{2}}{\bf B}_{\tau}^{-1}{\bf B}^{\frac{1}{2}}\preceq \sqrt{3}{\bf I}_{J}$.
		\end{description}
	\end{proposition}

	\begin{lemma}\label{spdneqs}\textnormal{(see, e.g., \cite{lin2021parallel})}
		Let ${\bf B}_1,{\bf B}_2\in\mathbb{R}^{k\times k}$ be real symmetric \textcolor{black}{matrices} such that ${\bf O}\prec{\bf B}_1\preceq{\bf B}_2$. Then, ${\bf O}\prec{\bf B}_2^{-1}\preceq{\bf B}_1^{-1}.$
	\end{lemma}

	\begin{lemma}\textnormal{\textcolor{black}{(see \cite{lin2018separable})}}\label{timematspd}
		For any $\alpha\in(0,1)$, it holds that $l_0^{(\alpha)}>0$ and ${\bf T}+{\bf T}^{\rm T}\succ{\bf O}$.
	\end{lemma}

	The following proposition holds obviously.
	\begin{proposition}\label{wghtsumbdlem}
		For positive numbers $\xi_i$ and $\zeta_i$ $(1\leq i\leq m)$,
		it obviously holds that
		\begin{equation*}
			\min\limits_{1\leq i\leq m}\frac{\xi_i}{\zeta_i}\leq\bigg(\sum\limits_{i=1}^{m}\zeta_i\bigg)^{-1}\bigg(\sum\limits_{i=1}^{m}\xi_i\bigg)\leq\max\limits_{1\leq i\leq m}\frac{\xi_i}{\zeta_i}.
		\end{equation*}
	\end{proposition}

	Recall that there is a factor $\frac{\sqrt{3}}{2}$ appearing in the definition of \eqref{btaudef}. Such a factor is for minimizing the condition number upper bound estimation in Theorem \ref{mainthm}. To see this, we denote
	\begin{align*}
		&{\bf B}_{\tau,\eta}:=\eta\sum\limits_{i=1}^{d}\eta_i{\bf I}_{m_i^{-}}\otimes \tau({\bf W}_{m_i}^{(\beta_i)})\otimes {\bf I}_{m_i^{+}},\quad {\bf P}_{\eta}:={\bf I}_{J}\otimes {\bf T}+{\bf B}_{\tau,\eta}\otimes{\bf I}_N,\\
		&{\bf P}_{l,\eta}:={\bf B}_{\tau,\eta}^{-\frac{1}{2}}\otimes {\bf T}+{\bf B}_{\tau,\eta}^{\frac{1}{2}}\otimes {\bf I}_{N},\quad {\bf P}_{r,\eta}={\bf B}_{\tau,\eta}^{\frac{1}{2}}\otimes {\bf I}_{N},
	\end{align*} 
	for some constant $\eta>0$.
	It is clear that 
	\begin{align}
		&{\bf B}_{\tau,\eta}=\frac{2\eta\sqrt{3}}{3}{\bf B}_{\tau},\label{tauetabtaurelate}\\
		&{\bf P}_{\eta}\stackrel{\eta=\sqrt{3}/2}{=\joinrel=\joinrel=\joinrel=\joinrel=\joinrel}{\bf P},\qquad {\bf P}_{l,\eta}\stackrel{\eta=\sqrt{3}/2}{=\joinrel=\joinrel=\joinrel=\joinrel=\joinrel}{\bf P}_{l},\qquad{\bf P}_{r,\eta}\stackrel{\eta=\sqrt{3}/2}{=\joinrel=\joinrel=\joinrel=\joinrel=\joinrel}{\bf P}_{r}\notag.
	\end{align}
	
	When $\eta=\frac{\sqrt{3}}{2}$, it holds ${\bf B}_{\tau,\eta}={\bf B}_{\tau}$, ${\bf P}_{\eta}={\bf P}$, ${\bf P}_{l,\eta}={\bf P}_{l}$ and ${\bf P}_{r,\eta}={\bf P}_{r}$. In what follows, we estimate an upper bound of $\kappa_2({\bf P}_{l,\eta}^{-1}{\bf A}{\bf P}_{r,\eta}^{-1})$ in terms of $\eta$, and the upper bound achieves minimal when $\eta=\frac{\sqrt{3}}{2}$.
	{
		\begin{theorem}\label{mainthm}
			Condition number of the preconditioned matrix ${\bf P}_{l,\eta}^{-1}{\bf A}{\bf P}_{r,\eta}^{-1}$ is uniformly bounded, i.e.,
			\begin{equation*}
				\sup\limits_{N,J}\kappa_2({\bf P}_{l,\eta}^{-1}{\bf A}{\bf P}_{r,\eta}^{-1})\leq \nu(\eta), 
			\end{equation*}
			with
			\begin{equation*}
				\nu(\eta):=\sqrt{\frac{3\max\left\{\frac{1}{2\eta},2\eta,1\right\}}{	\min\left\{\frac{1}{2\eta},\frac{2\eta}{3},1\right\}}}.
			\end{equation*}
			In particular, $\nu(\sqrt{3}/2)=\min\limits_{\eta\in(0,+\infty)}\nu(\eta)=3$, i.e.,
			\begin{equation*}
				\sup\limits_{N,J}\kappa_2({\bf P}_{l,\eta}^{-1}{\bf A}{\bf P}_{r,\eta}^{-1})\stackrel{\eta=\sqrt{3}/2}{=\joinrel=\joinrel=\joinrel=\joinrel=\joinrel}\sup\limits_{N,J}\kappa_2({\bf P}_{l}^{-1}{\bf A}{\bf P}_{r}^{-1})\leq \nu(\sqrt{3}/2)=3.
			\end{equation*}
		\end{theorem}
		\begin{proof} 	
			Denote $\hat{\bf A}={\bf B}^{\frac{1}{2}}\otimes{\bf I}_N+{\bf B}^{-\frac{1}{2}}\otimes{\bf T}$. Then, it is clear that 
			\begin{align}
				&{\bf A}=\hat{\bf A}({\bf B}^{\frac{1}{2}}\otimes{\bf I}_N),\notag\\
				&({\bf P}_{l,\eta}^{-1}{\bf A}{\bf P}_{r,\eta}^{-1})({\bf P}_{l,\eta}^{-1}{\bf A}{\bf P}_{r,\eta}^{-1})^{\rm T}={\bf P}_{l,\eta}^{-1}\hat{\bf A}[({\bf B}^{\frac{1}{2}}({\bf B}_{\tau,\eta})^{-1}{\bf B}^{\frac{1}{2}})\otimes{\bf I}_N]\hat{\bf A}^{\rm T}{\bf P}_{l,\eta}^{-\rm T}.\notag
			\end{align}
			By \eqref{tauetabtaurelate}, we have ${\bf B}^{\frac{1}{2}}({\bf B}_{\tau,\eta})^{-1}{\bf B}^{\frac{1}{2}}=\frac{\sqrt{3}}{2\eta}{\bf B}^{\frac{1}{2}}({\bf B}_{\tau})^{-1}{\bf B}^{\frac{1}{2}}$,
			which together with Proposition \ref{taumatneqs}${\bf (ii)}$ implies that
			\begin{equation*}
				\frac{1}{2\eta}{\bf I}_{J}	\preceq{\bf B}^{\frac{1}{2}}({\bf B}_{\tau,\eta})^{-1}{\bf B}^{\frac{1}{2}}\preceq \frac{3}{2\eta}{\bf I}_{J}.
			\end{equation*}
			That means
			\begin{align}\label{part1esti}
				\frac{1}{2\eta}{\bf P}_{l,\eta}^{-1}\hat{\bf A}\hat{\bf A}^{\rm T}{\bf P}_{l,\eta}^{-\rm T}\preceq({\bf P}_{l,\eta}^{-1}{\bf A}{\bf P}_{r,\eta}^{-1})({\bf P}_{l,\eta}^{-1}{\bf A}{\bf P}_{r,\eta}^{-1})^{\rm T}\preceq\frac{3}{2\eta}{\bf P}_{l,\eta}^{-1}\hat{\bf A}\hat{\bf A}^{\rm T}{\bf P}_{l,\eta}^{-\rm T}.
			\end{align}
			It thus remains to estimate Rayleigh quotient of ${\bf P}_{l,\eta}^{-1}\hat{\bf A}\hat{\bf A}^{\rm T}{\bf P}_{l,\eta}^{-\rm T}$. Let ${\bf z}\in\mathbb{R}^{ J N\times 1}$ \textcolor{black}{denote} any non-zero vector. Then,
			\begin{align}
				\frac{{\bf z}^{\rm T}{\bf P}_{l,\eta}^{-1}\hat{\bf A}\hat{\bf A}^{\rm T}{\bf P}_{l,\eta}^{-\rm T}{\bf z}}{{\bf z}^{\rm T}{\bf z}}&\stackrel{{\bf y}={\bf P}_{l,\eta}^{-\rm T}{\bf z}}{=\joinrel=\joinrel=\joinrel=\joinrel=\joinrel=}\frac{{\bf y}^{\rm T}\hat{\bf A}\hat{\bf A}^{\rm T}{\bf y}}{{\bf y}^{\rm T}{\bf P}_{l,\eta}{\bf P}_{l,\eta}^{\rm T}{\bf y}}\notag\\
				&=\frac{{\bf y}^{\rm T}[{\bf B}\otimes{\bf I}_N+{\bf I}_{ J }\otimes({\bf T}+{\bf T}^{\rm T})+{\bf B}^{-1}\otimes({\bf T}{\bf T}^{\rm T})]{\bf y}}{{\bf y}^{\rm T}[({\bf B}_{\tau,\eta})\otimes{\bf I}_N+{\bf I}_{ J }\otimes({\bf T}+{\bf T}^{\rm T})+({\bf B}_{\tau,\eta})^{-1}\otimes({\bf T}{\bf T}^{\rm T})]{\bf y}},\label{part2preesti}
			\end{align}
			By Lemma \ref{timematspd}, we know that ${\bf T}+{\bf T}^{\rm T}\succ{\bf O}$. Since ${\bf T}$ is a lower triangular matrix with its diagonal entries all equal to $l_0^{(\alpha)}>0$, ${\bf T}$ is invertible and thus ${\bf T}{\bf T}^{\rm T}\succ{\bf O}$. That means the matrices appearing in the numerator and \textcolor{black}{the denominator} of right hand side of \eqref{part2preesti} are all positive definite. Thus, Proposition \ref{wghtsumbdlem} is applicable to estimating \eqref{part2preesti}.
			
			\eqref{tauetabtaurelate} and Proposition \ref{taumatneqs}${\bf (i)}$ imply that
			\begin{equation}\label{btauetabrelate}
				{\bf O}\prec \frac{1}{2\eta}{\bf B}_{\tau,\eta}\preceq{\bf B}\preceq \frac{3}{2\eta}{\bf B}_{\tau,\eta}.
			\end{equation}
			Hence,
			\begin{equation}\label{firstterm}
				\frac{1}{2\eta}\leq\frac{{\bf y}^{\rm T}({\bf B}\otimes{\bf I}_N){\bf y}}{{\bf y}^{\rm T}[({\bf B}_{\tau,\eta})\otimes{\bf I}_N]{\bf y}}\leq \frac{3}{2\eta}.
			\end{equation}
			
			Lemma \ref{spdneqs} and \eqref{btauetabrelate} imply that
			\begin{equation*}
				{\bf O}\prec \frac{2\eta}{3}{\bf B}_{\tau,\eta}^{-1}\preceq {\bf B}^{-1}\preceq 2\eta {\bf B}_{\tau,\eta}^{-1}.
			\end{equation*}
			
			By \eqref{tauetabtaurelate}, Proposition \ref{taumatneqs}${\bf (i)}$ and Lemma \ref{spdneqs},
			\begin{align}
				\frac{2\eta}{3}=\frac{2\eta}{3}\times \frac{{\bf y}^{\rm T}[({\bf B}_{\tau,\eta})^{-1}\otimes({\bf T}{\bf T}^{\rm T})]{\bf y}}{{\bf y}^{\rm T}[({\bf B}_{\tau,\eta})^{-1}\otimes({\bf T}{\bf T}^{\rm T})]{\bf y}}&\leq\frac{{\bf y}^{\rm T}[{\bf B}^{-1}\otimes({\bf T}{\bf T}^{\rm T})]{\bf y}}{{\bf y}^{\rm T}[({\bf B}_{\tau,\eta})^{-1}\otimes({\bf T}{\bf T}^{\rm T})]{\bf y}}\notag\\
				&\leq 2\eta\times\frac{{\bf y}^{\rm T}[({\bf B}_{\tau,\eta})^{-1}\otimes({\bf T}{\bf T}^{\rm T})]{\bf y}}{{\bf y}^{\rm T}[({\bf B}_{\tau,\eta})^{-1}\otimes({\bf T}{\bf T}^{\rm T})]{\bf y}}=2\eta.\label{thirdterm}
			\end{align}
			Applying Proposition \ref{wghtsumbdlem} to \eqref{part2preesti}, \eqref{firstterm} and \eqref{thirdterm}, we obtain that
			\begin{equation*}
				\min\left\{\frac{1}{2\eta},\frac{2\eta}{3},1\right\}\leq\frac{{\bf z}^{\rm T}{\bf P}_{l,\eta}^{-1}\hat{\bf A}\hat{\bf A}^{\rm T}{\bf P}_{l,\eta}^{-\rm T}{\bf z}}{{\bf z}^{\rm T}{\bf z}}\leq\max\left\{\frac{1}{2\eta},2\eta,1\right\},
			\end{equation*}
			which together with \eqref{part1esti} implies that
			\begin{equation}\label{singvalsesti}
				\frac{1}{2\eta}\min\left\{\frac{1}{2\eta},\frac{2\eta}{3},1\right\}{\bf I}_{N J }\preceq({\bf P}_{l,\eta}^{-1}{\bf A}{\bf P}_{r,\eta}^{-1})({\bf P}_{l,\eta}^{-1}{\bf A}{\bf P}_{r,\eta}^{-1})^{\rm T}\preceq	\frac{3}{2\eta}\max\left\{\frac{1}{2\eta},2\eta,1\right\}{\bf I}_{N J }.
			\end{equation}
			\eqref{singvalsesti} implies that
			\begin{align*}
				\kappa_2({\bf P}_{l,\eta}^{-1}{\bf A}{\bf P}_{r,\eta}^{-1})&\leq\nu(\eta).
			\end{align*}
			
			$\nu(\cdot)$ is a single-variable function. It is easy to check that
			\begin{equation*}
				\nu(\sqrt{3}/2)=\min\limits_{\eta\in(0,+\infty)}\nu(\eta).
			\end{equation*}
			Recall that $${\bf P}_{l,\eta}^{-1}{\bf A}{\bf P}_{r,\eta}^{-1}\stackrel{\eta=\sqrt{3}/2}{=\joinrel=\joinrel=\joinrel=\joinrel=\joinrel}\sup\limits_{N,J}{\bf P}_{l}^{-1}{\bf A}{\bf P}_{r}^{-1}.$$
			Hence, $\sup\limits_{N,J}\kappa_2({\bf P}_{l}^{-1}{\bf A}{\bf P}_{r}^{-1})\leq \nu(\sqrt{3}/2)=3$.
			The proof is completed.
		\end{proof}
		
		\begin{remark}
			Theorem shows that the upper bound $\nu(\eta)$ of condition number of $\kappa_2({\bf P}_{l,\eta}^{-1}{\bf A}{\bf P}_{r,\eta}^{-1})$ achieves minimum at $\eta=\frac{\sqrt{3}}{2}$, which explains why we take a factor $\frac{\sqrt{3}}{2}$ in the definition of ${\bf B}_{\eta}$ given in \eqref{btaudef}. Although Theorem \ref{mainthm} shows that ${\bf P}_{l}^{-1}{\bf A}{\bf P}_{r}^{-1}$, numerical evidence in \cite{lin2021parallel} shows that ${\bf P}^{-1}{\bf A}$ is ill-conditioned. Then, Theorem \ref{residual} indicates that even for an ill-conditioned system, one may have a good convergence rate estimation for Krylov subspace solver  by relating the ill-conditioned problem to a well-conditioned problem via their associated Krylov subspaces. Hence, the significance of Theorem \ref{residual} is that it provides more flexibility for designing a preconditioner and estimating the convergence rate of the Krylov subspace solver.
		\end{remark}
	}
	\section{The Implementation}\label{implementsection}
	In this section, we propose a fast implementation of the Krylov subspace solver for solving the preconditioned system \eqref{aaosystem}. It suffices to implement the underlying matrix-vector product efficiently to implement a Krylov subspace solver. In other words, we will discuss in this section how to efficiently compute a matrix-vector product ${\bf P}^{-1}{\bf A}{\bf v}$ for an arbitrarily given vector ${\bf v}$. 
	
	In what follows, we discuss the fast computation of a matrix-vector product $\hat{\bf v}={\bf P}^{-1}{\bf A}{\bf v}$ for a given vector ${\bf v}\in\mathbb{R}^{NJ\times 1}$. The computation of $\hat{\bf v}$ can be divided into two steps.
	\begin{align}
		&{\rm Step~1}: \quad  {\rm Compute~} \dot{\bf v}={\bf A}{\bf v},\label{step1pv}\\
		&{\rm Step~2}: \quad  {\rm Compute~}\hat{\bf v}={\bf P}^{-1}\dot{\bf v}.\label{step2pv}
	\end{align}
	
	In step (\ref{step1pv}), notice that $\dot{\bf v}={\bf A}{\bf v}=({\bf B}\otimes{\bf I}_N){\bf v}+({\bf I}_J\otimes{\bf T}){\bf v}$. It is clear that ${\bf B}$ and ${\bf T}$ are both (multilevel) Toeplitz matrices. Hence, the computation of \eqref{step1pv} requires $\mathcal{O}(JN\log (JN))$ flops. 
	
	To see the fast implementation of step (\ref{step2pv}), we will utilize a block diagonalization form of ${\bf P}^{-1} $. By Equation \eqref{btaudiag} and Lemma \ref{btauspdprops}, we know that
	\begin{equation}\label{btaudiagreform}
		{\bf B}_{\tau}={\bf S}\tilde{\bf \Lambda}{\bf S},\quad \tilde{\bf \Lambda}=\frac{\sqrt{3}}{2}{\bf \Lambda}\succ {\bf O}.
	\end{equation}
	From the definition of ${\bf S}$, we know that ${\bf S}$ is real symmetric and orthogonal, i.e., ${\bf S}^{-1}={\bf S}^{\rm T}={\bf S}$.

	By Equation \eqref{btaudiagreform} and \eqref{pdef}, we know that
	\begin{equation}\label{plinvprinvexplicit}
		{\bf P}^{-1}=({\bf S}\otimes{\bf I}_N)({\bf I}_J\otimes{\bf T}+\tilde{\bf \Lambda}\otimes {\bf I}_N)^{-1}({\bf S}\otimes{\bf I}_N).
	\end{equation} 
	%
	Thus the  matrix-vector product $\hat{\bf v}={\bf P}^{-1}\dot{\bf v}$ for a given vector ${\bf v}\in\mathbb{R}^{NJ\times 1}$ can be fast computed as following stps,
	
	\begin{align*}
		&{\rm Step~2.1}: \quad  {\rm Compute~}\ddot{\bf v}=({\bf S}\otimes{\bf I}_N)\dot{\bf v}.\\
		&{\rm Step~2.2}: \quad  {\rm Compute~}\dddot{\bf v}=({\bf I}_J\otimes{\bf T}+\tilde{\bf \Lambda}\otimes {\bf I}_N)^{-1}\ddot{\bf v}.\\
		&{\rm Step~2.3}: \quad  {\rm Compute~}\hat{\bf v}=({\bf S}\otimes{\bf I}_N)\dddot{\bf v}.
	\end{align*}
	
	Steps 2.1, 2.3 can be fast implemented by multi-dimension fast sine transform, which requires $\mathcal{O}(NJ\log J)$ operations.
	It remains to discuss the computation of step 2.2. Rewrite $\tilde{\bf \Lambda}$ in \eqref{btaudiagreform} as
	\begin{equation*}
		\tilde{\bf \Lambda}=\diag(\tilde{\lambda}_i)_{i=1}^{J}.
	\end{equation*}
	Clearly, $\tilde{\lambda}_i$'s ($i=1,2,...,J$) are all positive numbers. Then, from Step 2.2, we see that
	\begin{equation}\label{step3decomp}
		\dddot{\bf v}=\textcolor{black}{\blockdiag}(\tilde{\bf T}_i^{-1})_{i=1}^{ J }\ddot{\bf v},
	\end{equation}
	where
	\begin{equation}\label{Ti}
		\tilde{{\bf T}_i}={\bf T}+\tilde{\lambda}_i{\bf I}_N,\quad i= 1,2,...,J.
	\end{equation}
	These $\tilde{\bf T}_i$'s are invertible lower triangular Toeplitz (ILTT) matrices. \cite{commenges1984fast} proposed a fast algorithm for fast inversion of ILTT matrices. With the fast inversion algorithm proposed in \cite{commenges1984fast}, step 2.2 can be fast computed within  $\mathcal{O}(JN\log N)$ flops. Summing over the operation cost for steps 2.1--2.3, we see that the computation of \eqref{step2pv} requires $\mathcal{O}(NJ\log (NJ))$ flops.
	
	Summing up the above discussion, for a given vector ${\bf v}\in\mathbb{R}^{NJ\times 1}$, the matrix-vector product ${\bf P}^{-1}{\bf v}$ can be fast computed within $\mathcal{O}(NJ\log(NJ))$ flops. 
	
	\begin{remark}
		Recall the two-sided preconditioning matrix in \eqref{twosidedpreclinsys}. As proposed in \cite{zhao2023bilateral,lin2021parallel}, the matrix-vector product of the two-sided preconditioned matrix is computed by ${\bf P}_l^{-1}({\bf A}({\bf P}_r^{-1}{\bf v}))$. Meanwhile, the matrix-vector product of our single-sided preconditioned matrix is computed by ${\bf P}^{-1}{\bf A}{\bf v}$.
		It is easy to see that the computation of a matrix vector product associated with ${\bf P}^{-1}=({\bf S}\otimes{\bf I}_N)({\bf I}_J\otimes{\bf T}+\tilde{\bf \Lambda}\otimes {\bf I}_N)^{-1}({\bf S}\otimes{\bf I}_N)$ is slightly faster than that of ${\bf P}_l^{-1}=({\bf S}\otimes{\bf I}_N)(\tilde{\bf \Lambda}^{-\frac{1}{2}}\otimes{\bf T}+\tilde{\bf \Lambda}^{\frac{1}{2}}\otimes {\bf I}_N)^{-1}({\bf S}\otimes{\bf I}_N)$, not to mention that the two-sided one requires an additional matrix-vector product associated with ${\bf P}_r^{-1}$. Therefore, each matrix-vector product of our single-sided preconditioning method requires less operations than that of the two-sided one. Moreover, Theorem \ref{residual} shows that GMRES solver for the single-sided preconditioned system converges faster than that for the two-sided one. That means our single-sided preconditioning method requires less computational time in total than that of the two-sided one. Indeed, numerical results in \textcolor{black}{Section \ref{sec:numerical}} show that our preconditioning method requires less computational than that of the two-sided one in the actual numerical tests.
	\end{remark}

	\section{Numerical examples}\label{sec:numerical}
	In this section, we test the performance of our proposed preconditioner. All numerical experiments are carried out using Octave on an HP EliteDesk 800 G5 Small Form Factor PC with Intel Core i5-8500 CPU @ 3.00GHz and 16GB RAM. 
	
	Numerical results for different discretization schemes presented in the Appendices are similar. Hence, we only present the result associated with the spatial discretization scheme \ref{sftgrwldwk} in this section.
	
	To demonstrate the effective of our proposed preconditioned GMRES solver, we will compare its performance with the two-sided preconditioned GMRES solver proposed in \cite{zhao2023bilateral} and GMRES solver without preconditioner. We adopt the notations GMRES-OS, GMRES-TS and GMRES-${\bf I}$ to represent the GMRES solver with single-sided preconditioner (our proposed), two-sided preconditioner \cite{zhao2023bilateral} and no preconditioner, respectively.

	In any case, the GMRES solver is implemented using the built-in functions in Octave. In details, for all the experiments presented in this section, the initial guess of GMRES is set as zero and its stopping criterion is set as $||{\bf r}_k||_2\leq 10^{-10}||{\bf r}_0||_2$, where ${\bf r}_k$ denotes the residual vector at $k$th GMRES iteration for $k\geq 1$ and ${\bf r}_0$ denotes the initial residual vector. The restarting number for GMRES is set as 20. 
	
	The computational time, denoted by "CPU" in the tables below, is counted in unit of second by the Octave built-in function \textbf{tic/toc}. The number of iterations of GMRES is denoted by ``Iter". In any case, when the number of iterations of GMRES solver is over 10000, we stop the GMRES iteration by hand and use the notation ``--" to represent its CPU. 
	
	Define the measure of error as
	\begin{equation*}
		{\rm Error}:=||{\bf u}_{*}-{\bf u}_{exact}||_{\infty},
	\end{equation*}
	where ${\bf u}_{*}$ denotes some iterative solution of the linear system \eqref{aaosystem}; ${\bf u}_{exaact}$ denotes the values of exact solution of the continuous problem \eqref{tsfde} on the time-space  grid.
	As the errors between the exact solution and the numerical solutions by the three solvers are always roughly the same, we only present ``Error" of GMRES-OS in each table below.

	\begin{example} \label{example_2D}
		{\rm	
			Consider the equations \eqref{tsfde} with 
			\begin{align*}
				d=2,\quad T=&1,\quad \Omega=(0,1)\times (0,1),\quad c_1=c_2=1,\quad \psi(x_1,x_2)\equiv 0,\\
				f(x_1, x_2, t)=&\frac{t^{\alpha+1}}{2 \cos (\beta_1 \pi / 2)}\left\{\left(\frac{2\left[x_1^{2-\beta_1}+(1-x_1)^{2-\beta_1}\right]}{\Gamma(3-\beta_1)}-\frac{12\left[x_1^{3-\beta_1}+(1-x_1)^{3-\beta_1}\right]}{\Gamma(4-\beta_1)}\right.\right. \\ &\left.\left.+\frac{24}{\Gamma(5-\beta_1)}\left[x_1^{4-\beta_1}+(1-x_1)^{4-\beta_1}\right]\right) x_2^{2}(1-x_2)^{2} \right\} \\ 
				&+\frac{ t^{\alpha+1}}{2 \cos (\beta_2 \pi / 2)}\left\{ \left(\frac{2\left[x_2^{2-\beta_2}+(1-x_2)^{2-\beta_2}\right]}{\Gamma(3-\beta_2)}-\frac{12\left[x_2^{3-\beta_2}+(1-x_2)^{3-\beta_2}\right]}{\Gamma(4-\beta_2)}\right.\right.\\
				&\left.\left.+\frac{24\left[x_2^{4-\beta_2}+(1-x_2)^{4-\beta_2}\right]}{\Gamma(5-\beta_2)}\right) x_1^{2}(1-x_1)^{2} \right\} \\
				&+\Gamma(\alpha+2) t x_1^{2}(1-x_1)^{2} x_2^{2}(1-x_2)^{2},
			\end{align*}
			the exact solution of which is given by $u(x,t)=t^{\alpha+1}x_1^2(1-x_1)^2x_2^2(1-x_2)^2$.
			
			Example \ref{example_2D} is discretized with $N$ temporal grid points and $M$ spatial grid points along each spatial direction (i.e., $M_1=M_2=M$). Numerical results of GMRES solver with different preconditioners for solving Example \ref{example_2D} are presented in Tables \ref{tab:example_diff_N}-\ref{tab:example_diff_M}. As GMRES without preconditioner is too time-consuming for Example \ref{example_2D}, we skip the results of unpreconditioned GMRES in Example \ref{example_2D}. Tables \ref{tab:example_diff_N}-\ref{tab:example_diff_M} show that  GMRES-OS and GMRES-TS have roughly the same iteration numbers while the CPU cost of  GMRES-OS is less than that of  GMRES-TS. The close iteration numbers of the two solvers demonstrates that GMRES-OS converges no slower than GMRES-TS, which supports Theorem \ref{residual}. Besides, each iteration of GMRES-OS requires less flops than GMRES-TS, which leads to GMRES-OS costing less computational time than GMRES-TS in total. It is also observed that the iteration numbers of GMRES-OS and GMRES-TS changes slightly and are always small no matter how $(\alpha,\beta_1,\beta)$, $N$ and $M$ change. That demonstrates the effectiveness and robustness of preconditioning technique based on $\tau$-matrix approximation.
			
			\begin{table}[H]
				\caption{Numerical results of GMRES with different preconditioners for Example \ref{example_2D} with $M=129$.}
				\label{tab:example_diff_N}
				\renewcommand{\arraystretch}{1.1}
				\setlength{\tabcolsep}{1.4em}
				\begin{center}
					\begin{tabular}{ccc|cc|cc}
						\hline
						\multirow{2}{*}{$(\alpha,\beta_1,\beta_2)$}&\multirow{2}{*}{$N$} &	\multirow{2}{*}{Error}&\multicolumn{2}{c|}{GMRES-OS}&\multicolumn{2}{c}{GMRES-TS} \\    
						\cline{4-7}
						&&&$\mathrm{Iter}$&$\mathrm{CPU(s)}$&$\mathrm{Iter}$&$\mathrm{CPU(s)}$\\
						\hline															
						\multirow{3}{*}{(0.1,1.1,1.1)}		&	$2^{6}$	&	3.01E-04	&	8	&	5.56	&	8	&	7.50	\\	
						&	$2^{7}$	&	3.01E-04	&	8	&	12.64	&	8	&	15.11	\\	
						&	$2^{8}$	&	3.01E-04	&	8	&	23.93	&	8	&	30.09	\\	
						\hline
						\multirow{3}{*}{(0.1,1.1,1.5)}		&	$2^{6}$	&	1.44E-04	&	8	&	5.55	&	8	&	7.49	\\	
						&	$2^{7}$	&	1.44E-04	&	8	&	11.79	&	8	&	15.16	\\	
						&	$2^{8}$	&	1.44E-04	&	8	&	22.33	&	8	&	30.39	\\	
						\hline
						\multirow{3}{*}{(0.1,1.1,1.9)}&	$2^{6}$	&	9.37E-05	&	7	&	5.38	&	7	&	6.79	\\	
						&	$2^{7}$	&	9.37E-05	&	7	&	10.86	&	7	&	13.92	\\	
						&	$2^{8}$	&	9.37E-05	&	7	&	19.85	&	7	&	27.34	\\	
						\hline
						\multirow{3}{*}{(0.1,1.5,1.5)}&	$2^{6}$	&	3.17E-05	&	8	&	5.59	&	8	&	7.46	\\	
						&	$2^{7}$	&	3.17E-05	&	8	&	13.07	&	8	&	15.20	\\	
						&	$2^{8}$	&	3.17E-05	&	8	&	24.44	&	8	&	30.28	\\	
						\hline
						\multirow{3}{*}{(0.1,1.5,1.9)}	&	$2^{6}$	&	1.26E-05	&	7	&	4.93	&	8	&	7.64	\\	
						&	$2^{7}$	&	1.26E-05	&	7	&	9.98	&	8	&	15.25	\\	
						&	$2^{8}$	&	1.26E-05	&	7	&	20.57	&	8	&	30.43	\\	
						\hline
						\multirow{3}{*}{(0.1,1.9,1.9)} &	$2^{6}$	&	4.92E-07	&	6	&	4.40	&	6	&	6.12	\\	
						&	$2^{7}$	&	4.92E-07	&	6	&	9.31	&	6	&	12.30	\\	
						&	$2^{8}$	&	4.92E-07	&	6	&	18.84	&	6	&	24.70	\\	
						\hline
						\multirow{3}{*}{(0.9,1.1,1.1)}	&	$2^{6}$	&	2.69E-04	&	9	&	6.11	&	9	&	8.25	\\	
						&	$2^{7}$	&	2.71E-04	&	9	&	13.10	&	9	&	16.67	\\	
						&	$2^{8}$	&	2.72E-04	&	9	&	24.05	&	9	&	33.26	\\	
						\hline
						\multirow{3}{*}{(0.9,1.1,1.5)}	&	$2^{6}$	&	1.30E-04	&	9	&	6.74	&	9	&	8.19	\\	
						&	$2^{7}$	&	1.31E-04	&	9	&	12.51	&	9	&	16.52	\\	
						&	$2^{8}$	&	1.32E-04	&	9	&	24.16	&	9	&	33.13	\\	
						\hline
						\multirow{3}{*}{(0.9,1.1,1.9)}&	$2^{6}$	&	8.64E-05	&	7	&	5.39	&	7	&	6.84	\\	
						&	$2^{7}$	&	8.76E-05	&	7	&	10.00	&	7	&	13.70	\\	
						&	$2^{8}$	&	8.81E-05	&	7	&	20.54	&	7	&	27.41	\\	
						\hline
						\multirow{3}{*}{(0.9,1.5,1.5)}&	$2^{6}$	&	2.71E-05	&	8	&	6.21	&	9	&	8.26	\\	
						&	$2^{7}$	&	2.85E-05	&	8	&	11.05	&	9	&	16.50	\\	
						&	$2^{8}$	&	2.91E-05	&	8	&	22.40	&	9	&	33.20	\\	
						\hline
						\multirow{3}{*}{(0.9,1.5,1.9)}&	$2^{6}$	&	1.02E-05	&	7	&	5.53	&	8	&	7.50	\\	
						&	$2^{7}$	&	1.12E-05	&	7	&	10.10	&	8	&	15.16	\\	
						&	$2^{8}$	&	1.17E-05	&	7	&	21.42	&	8	&	30.24	\\	
						\hline
						\multirow{3}{*}{(0.9,1.9,1.9)}	&	$2^{6}$	&	1.70E-06	&	6	&	4.85	&	6	&	6.06	\\	
						&	$2^{7}$	&	7.92E-07	&	6	&	10.15	&	6	&	12.17	\\	
						&	$2^{8}$	&	3.66E-07	&	6	&	18.58	&	6	&	32.07	\\	
						\hline
					\end{tabular}
				\end{center}
			\end{table}
			\begin{table}[H]
				\caption{Numerical results of GMRES with different preconditioners Example \ref{example_2D} with $N=128$ }
				\label{tab:example_diff_M}
				\renewcommand{\arraystretch}{1.1}
				\setlength{\tabcolsep}{1.4em}
				\begin{center}
					\begin{tabular}{ccc|cc|cc}
						\hline
						\multirow{2}{*}{$(\alpha,\beta_1,\beta_2)$}&\multirow{2}{*}{$M-1$} &	\multirow{2}{*}{Error}&\multicolumn{2}{c|}{GMRES-OS}&\multicolumn{2}{c}{GMRES-TS} \\    
						\cline{4-7}
						&&&$\mathrm{Iter}$&$\mathrm{CPU(s)}$&$\mathrm{Iter}$&$\mathrm{CPU(s)}$\\
						\hline															
						\multirow{3}{*}{(0.1,1.1,1.1)}		& $2^{6}$	&	5.51E-04	&	7	&	2.78	&	8	&	4.00	\\	
						&	$2^{7}$	&	3.01E-04	&	8	&	12.64	&	8	&	15.11	\\	
						&	$2^{8}$	&	1.58E-04	&	8	&	45.01	&	9	&	72.18	\\	
						\hline
						\multirow{3}{*}{(0.1,1.1,1.5)}&	$2^{6}$	&	2.70E-04	&	8	&	2.97	&	8	&	3.99	\\	
						&	$2^{7}$&	1.44E-04	&	8	&	11.79	&	8	&	15.16	\\	
						&	$2^{8}$&	7.41E-05	&	9	&	49.02	&	9	&	71.25	\\	
						\hline
						\multirow{3}{*}{(0.1,1.1,1.9)}	&	$2^{6}$	&	1.77E-04	&	6	&	2.88	&	7	&	3.68	\\	
						&	$2^{7}$	&	9.37E-05	&	7	&	10.86	&	7	&	13.92	\\	
						&	$2^{8}$	&	4.82E-05	&	7	&	40.48	&	8	&	65.40	\\	
						\hline
						\multirow{3}{*}{(0.1,1.5,1.5)}&	$2^{6}$	&	6.10E-05	&	7	&	2.47	&	8	&	4.03	\\	
						&	$2^{7}$	&	3.17E-05	&	8	&	13.07	&	8	&	15.20	\\	
						&	$2^{8}$	&	1.61E-05	&	8	&	44.42	&	9	&	72.30	\\	
						\hline
						\multirow{3}{*}{(0.1,1.5,1.9)}	&	$2^{6}$	&	2.36E-05	&	7	&	2.48	&	7	&	3.66	\\	
						&	$2^{7}$	&	1.26E-05	&	7	&	9.98	&	8	&	15.25	\\	
						&	$2^{8}$	&	6.53E-06	&	8	&	44.21	&	8	&	65.26	\\	
						\hline
						\multirow{3}{*}{(0.1,1.9,1.9)}&	$2^{6}$	&	6.77E-07	&	6	&	2.47	&	6	&	3.28	\\	
						&	$2^{7}$	&	4.92E-07	&	6	&	9.31	&	6	&	12.30	\\	
						&	$2^{8}$	&	4.14E-07	&	6	&	35.68	&	6	&	56.09	\\	
						\hline
						\multirow{3}{*}{(0.9,1.1,1.1)}	&	$2^{6}$	&	5.01E-04	&	8	&	2.71	&	8	&	3.98	\\	
						&	$2^{7}$	&	2.71E-04	&	9	&	13.10	&	9	&	16.67	\\	
						&	$2^{8}$	&	1.41E-04	&	9	&	53.91	&	9	&	72.10	\\	
						\hline
						\multirow{3}{*}{(0.9,1.1,1.5)}	&	$2^{6}$	&	2.49E-04	&	8	&	2.75	&	9	&	4.38	\\	
						&	$2^{7}$	&	1.31E-04	&	9	&	12.51	&	9	&	16.52	\\	
						&	$2^{8}$	&	6.68E-05	&	9	&	50.40	&	9	&	75.11	\\	
						\hline
						\multirow{3}{*}{(0.9,1.1,1.9)}&	$2^{6}$	&	1.67E-04	&	7	&	2.52	&	7	&	3.56	\\	
						&	$2^{7}$	&	8.76E-05	&	7	&	10.00	&	7	&	13.70	\\	
						&	$2^{8}$	&	4.45E-05	&	8	&	46.76	&	8	&	65.37	\\	
						\hline
						\multirow{3}{*}{(0.9,1.5,1.5)}	&	$2^{6}$	&	5.61E-05	&	8	&	2.71	&	9	&	4.38	\\	
						&	$2^{7}$	&	2.85E-05	&	8	&	11.05	&	9	&	16.50	\\	
						&	$2^{8}$	&	1.39E-05	&	9	&	49.60	&	9	&	69.97	\\	
						\hline
						\multirow{3}{*}{(0.9,1.5,1.9)}	&	$2^{6}$	&	2.17E-05	&	7	&	2.46	&	8	&	4.01	\\	
						&	$2^{7}$	&	1.12E-05	&	7	&	10.10	&	8	&	15.16	\\	
						&	$2^{8}$	&	5.34E-06	&	8	&	48.87	&	8	&	64.59	\\	
						\hline
						\multirow{3}{*}{(0.9,1.9,1.9)}&	$2^{6}$	&	1.03E-06	&	6	&	2.21	&	6	&	3.21	\\	
						&	$2^{7}$	&	7.92E-07	&	6	&	10.15	&	6	&	12.17	\\	
						&	$2^{8}$	&	3.68E-07	&	6	&	39.97	&	6	&	51.83	\\	
						\hline
					\end{tabular}
				\end{center}
			\end{table}
		}
	\end{example}
	
	\section{Conclusion}\label{sec:conclusion}
	This paper proposes a single-sided preconditioning method based on $\tau$-matrix approximation for a block triangular Toeplitz linear system arising from all-at-once of a non-local evolutionary equation with weakly singular kernels. Theoretically, we proved that the GMRES solver for the single-sided preconditioned system converges no slower than that for a auxiliary well-conditioned two-sided preconditioned system. Fast implementation is proposed for the matrix-vector multiplications associated with the preconditioned matrix. The theoretical results and the fast implementation are valid for several spatial discretization schemes. Numerical results reported have shown the efficiency and robustness of the proposed preconditioning technique.

	\section*{Acknowledgments}
	The work of Xue-lei  Lin was partially supported by research grants: 2021M702281 from China Postdoctoral Science Foundation, 12301480 from NSFC,  HA45001143 from Harbin Institute of Technology, Shenzhen, HA11409084  from Shenzhen. The work of Sean Hon was supported in part by the Hong Kong RGC under grant 22300921, a start-up grant from the Croucher Foundation, and a Tier 2 Start-up Grant from Hong Kong Baptist University.

	\bibliographystyle{plainnat}

	\begin{appendices}
		In the appendix, we verify $\{w_k\}_{k\geq 0}^{(\beta)}$ arising from spatial discretization schemes in the literature, which satisfy Property \ref{wkprop}.
		\section{Verification of $\{w_k\}_{k\geq 0}^{(\beta)}$ arising from \cite{ccelik2012crank}}
		$\{w_k\}_{k\geq 0}^{(\beta)}$ arising from \cite{ccelik2012crank} is defined by
		\begin{equation}\label{centraldiffwkdef}
			w_0^{(\beta)}=\frac{\Gamma(\beta+1)}{\Gamma(\beta/2+1)^2},\quad w_{k+1}^{(\beta)}=\left(1-\frac{\beta+1}{\beta/2+k+1}\right)w_{k}^{(\beta)},\quad k\geq 0.
		\end{equation}
		As $\Gamma(z)>0$ for $z>0$ and $\beta\in(1,2)$, it is clear that $w_0^{(\beta)}=\frac{\Gamma(\beta+1)}{\Gamma(\beta/2+1)^2}>0$ and that
		\begin{equation*}
			w_{1}^{(\beta)}=\left(1-\frac{\beta+1}{\beta/2+1}\right)w_{0}^{(\beta)}=\left(\frac{-\beta}{\beta+2}\right)w_0^{(\beta)}<0.
		\end{equation*}
		Notice that $1-\frac{\beta+1}{\beta/2+k+1}=\frac{2k-\beta}{2k+\beta+2}>0$ for $k\geq 1$. Therefore, it is trivial to see by induction that
		\begin{equation*}
			w_{k+1}^{(\beta)}=\left(1-\frac{\beta+1}{\beta/2+k+1}\right)w_{k}^{(\beta)}<0,\quad k\geq 1.
		\end{equation*}
		Hence, Property \ref{wkprop}${\bf (i)}$ is valid. Moreover, $1-\frac{\beta+1}{\beta/2+k+1}=\frac{2k-\beta}{2k+\beta+2}\in(0,1)$ for $k\geq 1$, 
		\begin{equation*}
			|w_{k+1}^{(\beta)}|=\left|\left(1-\frac{\beta+1}{\beta/2+k+1}\right)w_{k}^{(\beta)}\right|=\left|\frac{2k-\beta}{2k+\beta+2}\right||w_{k}^{(\beta)}|<|w_{k}^{(\beta)}|, \quad k\geq 1,
		\end{equation*}
		which combined with $w_{k}^{(\beta)}<0$ for $k\geq 1$ implies that $w_{k}^{(\beta)}<w_{k+1}^{(\beta)}$ for $k\geq 1$. In other words, Property \ref{wkprop}${\bf (iii)}$ is valid. 
		
		It is indicated in \cite{ccelik2012crank} that $2\sum\limits_{k=1}^{\infty}|w_k^{(\beta)}|=w_0^{(\beta)}$, which together with $w_k^{(\beta)}<0$ for $k\geq 1$ implies that
		\begin{align*}
			w_0^{(\beta)}+2\sum\limits_{k=1}^{m-1}w_k^{(\beta)}=-2\sum\limits_{k=1}^{\infty}w_k^{(\beta)}+2\sum\limits_{k=1}^{m-1}w_k^{(\beta)}=-2\sum\limits_{k=m}^{\infty}w_k^{(\beta)}=2\sum\limits_{k=m}^{\infty}|w_k^{(\beta)}|,\quad m\geq 1
		\end{align*}
		It is also shown in \cite{ccelik2012crank} that $|w_k^{(\beta)}|=\mathcal{O}((k+1)^{-\beta-1})$, which means
		\begin{align*}
			w_0^{(\beta)}+2\sum\limits_{k=1}^{m-1}w_k^{(\beta)}&=2\sum\limits_{k=m}^{\infty}|w_k^{(\beta)}|\\
			&\gtrsim \sum\limits_{k=m}^{\infty}\frac{1}{(k+1)^{\beta+1}}\geq \int_{m}^{\infty}\frac{1}{(1+x)^{1+\beta}}dx=\frac{1}{\beta(1+m)^{\beta}},\quad m\geq 1.
		\end{align*}
		Hence,
		\begin{equation*}
			(m+1)^{\beta}\left(w_0^{(\beta)}+2\sum\limits_{k=1}^{m-1}w_k^{(\beta)}\right)\gtrsim \frac{1}{\beta}>0,\quad m\geq 1.
		\end{equation*}
		Therefore,
		\begin{equation*}
			\inf\limits_{m\geq 1}(m+1)^{\beta}\left(w_0^{(\beta)}+2\sum\limits_{k=1}^{m-1}w_k^{(\beta)}\right)>0,
		\end{equation*}
		which means Property \ref{wkprop}${\bf (ii)}$ is valid.
		
		\section{Verification of $\{w_k^{(\beta)}\}_{k=0}^{\infty}$ arising from \cite{meerschaert2006finite}}
		$\{w_k^{(\beta)}\}_{k=0}^{\infty}$ arising from \cite{meerschaert2004finite} is defined by 
		\begin{align}
			&w_k^{(\beta)}=\gamma_{\beta}\tilde{w}_k^{(\beta)},\quad k\geq 0,\qquad \gamma_{\beta}=\frac{-1}{2\cos(\beta\pi/2)}>0,\label{sftgrwldwk}\\
			&\tilde{w}_k^{(\beta)}=2g_1^{(\beta)},\quad \tilde{w}_1^{(\beta)}=g_0^{(\beta)}+g_2^{(\beta)},\quad \tilde{w}_k^{(\beta)}=g_{k+1}^{(\beta)},~k\geq 2,\notag\\
			&g_0^{(\beta)}=-1,\quad g_{k+1}^{(\beta)}=\left(1-\frac{\beta+1}{k+1}\right)g_k^{(\beta)},~k\geq 0.\notag
		\end{align}
		Property \ref{wkprop}${\bf (i)}$, ${\bf (iii)}$ of $\{w_k^{(\beta)}\}_{k=0}^{\infty}$ defined in \eqref{sftgrwldwk} has been verified in  \cite[Lemma 4.1]{huang2021spectral}. 
		
		It thus remains to verify Property \ref{wkprop}${\bf (ii)}$. By \cite[Lemma 8]{lin2018stability}, it holds
		\begin{equation}\label{gkprop}
			\sum\limits_{k=0}^{\infty}g_k^{(\beta)}=0.
		\end{equation}
		Therefore,
		\begin{align*}
			w_0^{(\beta)}+2\sum\limits_{k=1}^{\infty}w_k^{(\beta)}&=\gamma_{\beta}\left(\tilde{w}_0^{(\beta)}+2\sum\limits_{k=1}^{\infty}\tilde{w}_k^{(\beta)}\right)\\
			&=\gamma_{\beta}\left[2g_1^{(\beta)}+2(g_0^{(\beta)}+g_2^{(\beta)})+2\sum\limits_{k=2}^{\infty}g_{k+1}^{(\beta)}\right]\\
			&=2\gamma_{\beta}\sum\limits_{k=0}^{\infty}g_k^{(\beta)}=0,
		\end{align*}
		which together with Property \ref{wkprop}${\bf (i)}$ implies that
		\begin{align*}
			w_0^{(\beta)}+2\sum\limits_{k=1}^{m-1}w_k^{(\beta)}=-2\sum\limits_{k=1}^{\infty}w_k^{(\beta)}+2\sum\limits_{k=1}^{m-1}w_k^{(\beta)}=-2\sum\limits_{m}^{\infty}w_k^{(\beta)}=2\sum\limits_{k=m}^{\infty}|w_k^{(\beta)}|.
		\end{align*}
		It is shown in \cite{meerschaert2004finite} that $g_k^{(\beta)}=\mathcal{O}((k+1)^{-\beta-1})$ for $k\geq 0$,  Therefore, $|w_k^{(\beta)}|=\mathcal{O}((k+1)^{-\beta-1})$, which means
		\begin{align*}
			w_0^{(\beta)}+2\sum\limits_{k=1}^{m-1}w_k^{(\beta)}&=2\sum\limits_{k=m}^{\infty}|w_k^{(\beta)}|\\
			&\gtrsim \sum\limits_{k=m}^{\infty}\frac{1}{(k+1)^{\beta+1}}\geq \int_{m}^{\infty}\frac{1}{(1+x)^{1+\beta}}dx=\frac{1}{\beta(1+m)^{\beta}},\quad m\geq 1.
		\end{align*}
		Hence,
		\begin{equation*}
			(m+1)^{\beta}\left(w_0^{(\beta)}+2\sum\limits_{k=1}^{m-1}w_k^{(\beta)}\right)\gtrsim \frac{1}{\beta}>0,\quad m\geq 1.
		\end{equation*}
		Therefore,
		\begin{equation*}
			\inf\limits_{m\geq 1}(m+1)^{\beta}\left(w_0^{(\beta)}+2\sum\limits_{k=1}^{m-1}w_k^{(\beta)}\right)>0.
		\end{equation*}
		Thus, Property \ref{wkprop}${\bf (ii)}$ of $\{w_k^{(\beta)}\}_{k=0}^{\infty}$ defined in \eqref{sftgrwldwk}  is valid.
		
		\section{Verification of $\{w_k^{(\beta)}\}_{k=0}^{\infty}$ arising from \cite{sousaelic}}
		$\{w_k^{(\beta)}\}_{k=0}^{\infty}$ arising from \cite{sousaelic} is defined by 
		\begin{align}
			&w_k^{(\beta)}=\hat{\gamma}_{\beta}\hat{w}_k^{(\beta)},\quad k\geq 0,\qquad \hat{\gamma}_{\beta}=\frac{-1}{2\cos(\beta\pi/2)\Gamma(4-\beta)}>0,\label{wghtfdwk}\\
			&\hat{w}_0^{(\beta)}=2p_1^{(\beta)},\quad \hat{w}_1^{(\beta)}=p_0^{(\beta)}+p_2^{(\beta)},\quad \hat{w}_k^{(\beta)}=p_{k+1}^{(\beta)},~k\geq 2,\notag\\
			&p_0^{(\beta)}=-1,\quad p_{1}^{(\beta)}=4-2^{3-\beta},\quad p_2^{(\beta)}=-3^{3-\beta}+4\times2^{3-\beta}-6,\notag\\
			&p_k^{(\beta)}=-(k+1)^{3-\beta}+4k^{3-\beta}-6(k-1)^{3-\beta}+4(k-2)^{3-\beta}-(k-3)^{3-\beta},\quad k\geq 3.\notag
		\end{align}
		It is easy to see that $w_0^{(\beta)}=\hat{\gamma}_{\beta}\hat{w}_0^{(\beta)}=2\hat{\gamma}_{\beta}p_1^{(\beta)}>0$ and that
		\begin{equation*}
			w_1^{(\beta)}=\hat{\gamma}_{\beta}(p_0^{(\beta)}+p_2^{(\beta)})=\hat{\gamma}_{\beta}(-3^{3-\beta}+4\times2^{3-\beta}-7)<0,\quad \beta\in(1,2).
		\end{equation*} 
		Moreover, it is shown in \cite[Lemma 4]{sousaelic} that $p_k^{(\beta)}\leq 0$ for $k\geq 3$. Then, $w_k^{(\beta)}=\hat{\gamma}_{\beta}\hat{w}_k^{(\beta)}=\hat{\gamma}_{\beta}p_{k+1}^{(\beta)}\leq 0$ for $k\geq 2$. So far, Property \ref{wkprop}${\bf (i)}$ of $\{w_k^{(\beta)}\}_{k=0}^{\infty}$  defined in \eqref{wghtfdwk} is shown to be valid.
		
		Notice that
		\begin{align*}
			\hat{w}_1^{(\beta)}-\hat{w}_2^{(\beta)}&=p_0^{(\beta)}+p_2^{(\beta)}-p_3^{(\beta)}\\
			&=4^{3-\beta}-5\times 3^{3-\beta}+10\times 2^{3-\beta}-11\leq 0,\quad \beta\in(1,2).
		\end{align*}
		In other words, $w_1^{(\beta)}=\hat{\gamma}_{\beta}\hat{w}_1^{(\beta)}\leq \hat{\gamma}_{\beta}\hat{w}_2^{(\beta)}=w_2^{(\beta)}$. Moreover, it is shown in \cite[Lemma  4]{sousaelic} that $p_k^{(\beta)}\leq p_{k+1}^{(\beta)}$ for $k\geq 3$. Thus, $w_k^{(\beta)}=\hat{\gamma}_{\beta}p_{k+1}^{(\beta)}\leq\hat{\gamma}_{\beta}p_{k+3}^{(\beta)}=w_{k+1}^{(\beta)}$ for $k\geq 2$, which means  Property \ref{wkprop}${\bf (iii)}$ of $\{w_k^{(\beta)}\}_{k=0}^{\infty}$  defined in \eqref{wghtfdwk} is valid.
		
		It thus remains to verify Property \ref{wkprop}${\bf (ii)}$. By \cite[Lemma 4]{sousaelic}, it holds
		\begin{equation}\label{gkprop}
			\sum\limits_{k=0}^{\infty}p_k^{(\beta)}=0.
		\end{equation}
		Therefore,
		\begin{align*}
			w_0^{(\beta)}+2\sum\limits_{k=1}^{\infty}w_k^{(\beta)}&=\hat{\gamma}_{\beta}\left(\hat{w}_0^{(\beta)}+2\sum\limits_{k=1}^{\infty}\hat{w}_k^{(\beta)}\right)\\
			&=\hat{\gamma}_{\beta}\left[2p_1^{(\beta)}+2(p_0^{(\beta)}+p_2^{(\beta)})+2\sum\limits_{k=2}^{\infty}p_{k+1}^{(\beta)}\right]\\
			&=2\hat{\gamma}_{\beta}\sum\limits_{k=0}^{\infty}p_k^{(\beta)}=0,
		\end{align*}
		which together with Property \ref{wkprop}${\bf (i)}$ implies that
		\begin{align*}
			w_0^{(\beta)}+2\sum\limits_{k=1}^{m-1}w_k^{(\beta)}=-2\sum\limits_{k=1}^{\infty}w_k^{(\beta)}+2\sum\limits_{k=1}^{m-1}w_k^{(\beta)}=-2\sum\limits_{m}^{\infty}w_k^{(\beta)}=2\sum\limits_{k=m}^{\infty}|w_k^{(\beta)}|
		\end{align*}
		It is shown in \cite{lin2018stability} that $p_k^{(\beta)}=\mathcal{O}((k+1)^{-\beta-1})$ for $k\geq 0$,  Therefore, $|w_k^{(\beta)}|=\mathcal{O}((k+1)^{-\beta-1})$, which means
		\begin{align*}
			w_0^{(\beta)}+2\sum\limits_{k=1}^{m-1}w_k^{(\beta)}&=2\sum\limits_{k=m}^{\infty}|w_k^{(\beta)}|\\
			&\gtrsim \sum\limits_{k=m}^{\infty}\frac{1}{(k+1)^{\beta+1}}\geq \int_{m}^{\infty}\frac{1}{(1+x)^{1+\beta}}dx=\frac{1}{\beta(1+m)^{\beta}},\quad m\geq 1.
		\end{align*}
		Hence,
		\begin{equation*}
			(m+1)^{\beta}\left(w_0^{(\beta)}+2\sum\limits_{k=1}^{m-1}w_k^{(\beta)}\right)\gtrsim \frac{1}{\beta}>0,\quad m\geq 1.
		\end{equation*}
		Therefore,
		\begin{equation*}
			\inf\limits_{m\geq 1}(m+1)^{\beta}\left(w_0^{(\beta)}+2\sum\limits_{k=1}^{m-1}w_k^{(\beta)}\right)>0.
		\end{equation*}
		Thus, Property \ref{wkprop}${\bf (ii)}$ of $\{w_k^{(\beta)}\}_{k=0}^{\infty}$ defined in \eqref{wghtfdwk}  is valid.
	\end{appendices}
\end{document}